\documentclass[12pt]{article}
\usepackage[T1]{fontenc}
\usepackage[english]{babel}
\usepackage{amsmath,amssymb,amsthm}
\usepackage{amscd}
\usepackage{mathrsfs} %pour mathscr
\usepackage[shortlabels]{enumitem}
\usepackage{euscript}
\RequirePackage[colorlinks,citecolor=blue,urlcolor=blue]{hyperref}

\newtheorem{theorem}{Theorem}[section]
\newtheorem{corollary}[theorem]{Corollary}

\newtheorem{fact}[theorem]{Fact}
\newtheorem{proposition}[theorem]{Proposition}

\theoremstyle{definition}

\newtheorem{remark}[theorem]{Remark}
\newtheorem{example}[theorem]{Example}

\numberwithin{equation}{section}

\newcommand{\abs}[1]{{\lvert {#1}\rvert}}
\newcommand{\norm}[1]{{\lVert {#1}\rVert}}

\newcommand{\scal}[2]{{\left\langle{{#1}\,\vert\,{#2}}\right\rangle}}
\newcommand{\pair}[2]{{\left\langle{{#1},{#2}}\right\rangle}}

\DeclareMathOperator{\argmin}{argmin}

\DeclareMathOperator{\dom}{dom}
\DeclareMathOperator{\sign}{sign}

\newcommand{\bb}{\mathcal{B}}
\newcommand{\XC}{\mathcal{X}}
\newcommand{\XD}{\mathcal{D}}
\newcommand{\YC}{\mathcal{Y}}

\newcommand{\HH}{\ensuremath{{\EuScript{H}}}}
\newcommand{\FF}{\ensuremath{{\EuScript{F}}}}

\newcommand{\B}{\mathcal{B}}
\newcommand{\F}{\mathcal{F}}
\newcommand{\PP}{\ensuremath{ P}}
\newcommand{\loss}{\ensuremath{L}}
\newcommand{\prox}{\mathrm{prox}}

\newcommand{\R}{\mathbb{R}}
\newcommand{\C}{\mathbb{C}}
\newcommand{\KK}{\mathbb{K}}
\newcommand{\XX}{\ensuremath{{\mathsf X}}}
\newcommand{\KKK}{\ensuremath{{\mathsf K}}}
\newcommand{\RP}{\ensuremath{{\mathbb R}_+}}
\newcommand{\RX}{\ensuremath{\left]-\infty, +\infty\right]}}
\newcommand{\RPP}{\ensuremath{{\mathbb R}_{++}}}
\newcommand{\N}{\mathbb{N}}
\newcommand{\pinf}{\ensuremath{{+\infty}}}
\newcommand{\ud}{\ensuremath{\,{\mathrm d}}}
\newcommand{\emp}{\ensuremath{{\varnothing}}}

\newcommand{\ii}{j}

\begin{document}

\title{ {\bf 
%The dual approach to support vector machines
%General support vector regression
%in Banach spaces through a dual tensor-kernel representation.
Generalized support vector regression:
duality and tensor-kernel representation.
% Tensor Kernel machines (TKM)
}} 
%under weakened differentiability assumptions}}

\author{Saverio Salzo$^1$ and Johan A.K. Suykens$^2$\\[3mm]
\small
\small $\!^1$LCSL, Istituto Italiano di Tecnologia and
Massachusetts Institute of Technology\\
\small Via Morego 30, 16163 Genova, Italy\\
%\small Bldg.~46-5155, 77 Massachusetts Avenue, Cambridge, MA 02139, USA\\
\small Email: saverio.salzo@iit.it\\[3mm]
\small
\small $\!^2$KU Leuven, ESAT-STADIUS\\
\small Kasteelpark Arenberg 10, B-3001 Leuven (Heverlee), Belgium \\
\small Email: johan.suykens@esat.kuleuven.be
}

\date{}
\maketitle
{\abstract 
In this paper we study the variational problem associated to 
support vector regression in Banach function spaces. 
Using the Fenchel-Rockafellar duality theory, we give
explicit formulation of the dual problem as well as
of the related optimality conditions. Moreover, we 
provide a new computational framework for solving the problem 
which relies on a tensor-kernel representation. This analysis 
overcomes the typical difficulties connected to learning in Banach spaces. 
We finally present a large class of tensor-kernels to which
our theory fully applies: power series tensor kernels.
This type of kernels describe Banach spaces of analytic functions
and include generalizations of the  exponential and polynomial kernels 
as well as, in the complex case, generalizations of the Szeg\"o
and Bergman kernels.
}


\section{Introduction}

Support vector regression is a kernel-based estimation
technique which allows to estimate a function belonging to 
an infinite dimensional function space based on a finite number
of pointwise observations \cite{Cri00,SteChi2008,Suy2002,Vap98}. 
The (primal) problem is classically formulated as  an empirical risk minimization
 on a reproducing kernel Hilbert space of functions,
the regularization term being the square of the Hilbert norm.
This infinite dimensional optimization problem
is approached through its dual problem which turns out to be
finite dimensional, quadratic (possibly constrained), and 
involving the kernel function only, 
evaluated at the available data points \cite{Cri00,SteChi2009,Vap98}.
Therefore, the knowledge of the kernel suffices to completely
describe and solve the dual problem as well as to compute the solution 
of the primal (infinite dimensional) problem.
This is what it is known as the \emph{kernel trick} and makes
support vector regression effective and so popular in applications \cite{SteChi2008}.
 
Learning in Banach spaces of functions is an emerging area of research 
which in principle permits to consider learning problems with more general 
types of norms than Hilbert norms \cite{Com15,Fass15,Zhang2009}.
The main motivation for this generalization comes from the 
need of finding more effective sparse representations of data
or for feature selection.
To that purpose, several types of alternative regularization schemes have been
proposed in the literature, and we mention, among  others, 
$\ell^1$ regularization (lasso), elastic net, and bridge regression \cite{Demo09,Fu1998}.
Moreover, the statistical consistency of such more general regularization schemes
have been addressed in \cite{Com15,Com15b,Demo09,Kol2009}.
However, moving to Banach spaces of functions
and Banach norms pose serious difficulties from the computational point of view
\cite{Sri11}. Indeed, even though, in this more general setting, 
it is still possible to introduce appropriate
reproducing kernels \cite{Zhang2009}, they fail to properly
represent the solution of the dual and primal problem,
so that the dual approach becomes cumbersome.
For this reason, the above mentioned estimation techniques
are often implemented by directly tackling  the primal problem and therefore, 
as a matter of fact, reduces to a finite dimensional estimation methods
(that is to parametric models).

In this work we address support vector regression in Banach function spaces 
and we provide a new computational framework for solving the associated 
optimization problem,
overcoming the difficulties we discussed above.
Our model is described in the primal
by means of an appropriate feature map in Banach spaces
of features and a general regularizer. 
We first study, in great generality,
the interplay between 
the primal and the dual problem
through the Fenchel-Rockafellar duality.
We obtain an explicit formulation of the 
dual problem, as well as of the related optimality
conditions, in terms of the feature map and the subdifferentials of the 
loss function and of the regularizer. As a byproduct we also provide a
general representer theorem.

Next, we consider the setting of a linear model described through
a countable dictionary of functions with the
regularization term being a function of the $\displaystyle\ell^r$-norm of the related coefficients,
 with $r=m/(m-1)$ and $m$ an even integer. This choice 
 allows $r>1$ to be close to $1$ and hence to 
 approximate $\ell^1$ regularization, possibly 
keeping the stability properties of the $\ell^2$ regularization
based estimation. 
Then we introduce a new type of kernel function
which turns to be a symmetric positive-definte tensor of order $m$,
and we prove that it
allows to formulate the dual problem without any reference
to the underlying feature map as well as to evaluate the optimal
regression function at any point in the input space.
In this way, the dual problem becomes a finite dimensional 
convex homogeneous $m$-degree-polynomial minimization problem
which can be solved by standard smooth optimization algorithms,
e.g., Newton-type methods.
In the end, we show that the kernel trick
can be fully extended to \emph{tensor-kernels} and makes
the dual approach in the Banach setting still viable for computing the solution of 
the primal (infinite dimensional) problem. 
Finally, we illustrate the theoretical framework above by
presenting an entire class of tensor-kernel
functions, that is  \emph{power series tensor-kernels},
which are extensions of the analogue matrix-type power series kernels considered in \cite{Zwi2009}. We show that this
class includes kernels of exponential and polynomial
type as well as, in the complex case, generalizations of the Szeg\"o
and Bergman kernels. 

The rest of the paper is organized as follows. Section~\ref{sec:basic}
gives basic definitions and facts. Section~\ref{sec:GenSVR}
presents the dual framework for SVR in general Banach
spaces of features. In Section~\ref{sec:kernrep} we introduce
 tensor kernels and explain their role in making 
 Banach space problems more practical numerically.
 Section~\ref{secPSTK} treats tensor kernels of power series type,
 which give rise to a general class of function Banach spaces to
 which the theory applies.
 Finally Section~\ref{sec:con} contains conclusions.

\section{Basic definitions and facts}
\label{sec:basic}

Let $\F$ be a real  Banach space. We denote by $\F^*$ its dual space 
and by $\pair{\cdot}{\cdot}$ the canonical pairing between $\F$ and $\F^*$,
meaning that, for every $(w,w^*) \in \F\times\F^*$, $\pair{w}{w^*} 
= w^*(w)$. We denote by $\norm{\cdot}$ the norm of $\FF$ as well as
the norm of $\FF^*$.
Let $F\colon\F\to\RX$. The \emph{domain of $F$} is 
$\dom F=\{ w\in \F~\vert~F(w)<\pinf\}$
and $F$ is 
\emph{proper} if $\dom F\neq\emp$. 
Suppose that $F$ is proper and convex.
The \emph{subdifferential} of $F$ is the 
set-valued operator $\partial F\colon\F\to 2^{\F^*}$ such that,
\begin{equation*}
(\forall\, w\in \F)\quad 
\partial F(w) =\big\{  w^*\in \F^*~\big\vert~(\forall v\in\F)\; F(w)+\pair{v-w}{w^*}\leq F(v) \big\},
\end{equation*}
and its domain is $\dom\partial F=
\{w\in \F~\vert~\partial F(w)\neq\emp \}$.
The \emph{Fenchel conjugate} of $F$ is the function
$F^*\colon \F^* \to \RX\colon w^* \in \F^* \mapsto \sup_{w \in \F} \pair{w}{w^*} - F(w)$.
We denote by $\Gamma_0(\F)$ the set of proper, convex, and lower semicontinuous
functions on $\F$. 
If $C \subset \F$, we denote by $\iota_C$ the 
\emph{indicator function} of $C$,
that is $\iota_C \colon \F \to \RX$, such that, for every $w \in \F$,
$\iota_C(w) = 0$ if $w \in C$, and $\iota_C(w) = \pinf$ if $w \notin C$.
Let $F \in \Gamma_0(\F)$. Then the following  
duality relation between the subdifferentials of $F$ and its conjugate $F^*$ holds
\cite[Theorem~2.4.4(iv)]{Zali02}
\begin{equation}
\label{eq:20160302a}
(\forall\, (w,w^*) \in \F\times\F^*)\qquad w^* \in \partial F(w)\ \Leftrightarrow\ w \in \partial F^*(w^*).
\end{equation}

Let $r \in \left[1,+\infty\right[$. The \emph{conjugate} exponent
of $r$ is $r^* \in \left]1,+\infty\right]$ such that $1/r + 1/r^* = 1$.
If $(\mathcal{Z},\mathfrak{A}, \mu)$ is a finite measure space, 
we denote by $\pair{\cdot}{\cdot}_{r,r^*}$ the canonical pairing between
the Lebesgue spaces
$\displaystyle L^r(\mu)$ and $L^{r^*}\!(\mu)$, i.e., $\pair{f}{g}_{r,r^*} = \int_{\mathcal{Z}} f g \ud \mu$.
If $\KK$ is a countable set, we define the sequence space
\begin{equation*}
\ell^r(\KK) = \bigg\{(w_k)_{k \in \KK} \in 
\R^{\KK}~\bigg\vert~\sum_{k \in \KK}\abs{w_k}^r<+\infty\bigg\}
\end{equation*}
endowed with the norm $\norm{w}_r = \big( \sum_{k \in \KK}\abs{w_k}^r \big)^{1/p}$.
The pairing between $\displaystyle\ell^r(\KK)$ and 
$\displaystyle\ell^{r^*}\!(\KK)$ is 
$\pair{w}{w^*}_{r,r^*} = \sum_{k \in \KK} w_k w_k^*$.

The Banach space $\F$ is called \emph{smooth} \cite{Ciora90} if,
for every $w \in \F$ there exists a unique $w^* \in \F^*$ such that 
$\norm{w^*}=1$ and $\pair{w}{w^*} = 1$. The smoothness property
is equivalent to the G\^ateaux differentiability of the norm on $\F\setminus\{0\}$.
We say that $\F$ is \emph{strictly convex} if, for every $w$ and every $v$ in $\F$
such that $\norm{w}=\norm{v}=1$ and $w\neq v$, one has $\norm{(w+v)/2} <1$.
Let $\F$ be a reflexive, strictly convex and smooth real Banach space
and let $r \in \left]1,+\infty\right[$.
Then the \emph{$r$-duality map} of $\F$  is the mapping \cite{Ciora90}
\begin{equation}
\label{e:JJJ}
J_r \colon \F \to \F^*\text{ such that }(\forall\, w \in \F)
\quad \pair{w}{J_r(w)}=\norm{w}^r\quad \text{and}\quad
\norm{J_r(w)}=\norm{w}^{r-1}.
\end{equation}
This map is a bijection from $\F$ onto $\F^*$
and its inverse is the $r^*$-duality map of $\F^*$.
Moreover, for every $w \in \F$ and every $\lambda \in \RP$, $J_{r}(\lambda w) 
= \lambda^{r-1} J_{r}(w)$ and $J_{r}(-w) = - J_{r}(w)$.
The mapping $J_2$ is called the \emph{normalized duality map} of $\F$.
The Banach space $\displaystyle\ell^r(\KK)$
is reflexive, strictly convex, and smooth, and, it is
immediate to verify from \eqref{e:JJJ} that, its $r$-duality map is
\begin{equation}
\label{eq:dualitymapellr}
J_{r}\colon \ell^r(\KK)\to \ell^{r^*}\!(\KK)\colon w=(w_k)_{k\in\KK}
\mapsto 
(\abs{w_k}^{r-1} \sign(w_k))_{k\in\KK}.
\end{equation}
Moreover, $\displaystyle J_{r}^{-1}\colon \ell^{r^*}\!(\KK) \to \ell^{r}(\KK)$ is
the ${r^*}$-duality map of $\displaystyle \ell^{r^*}\!(\KK)$, hence it has the
same form as \eqref{eq:dualitymapellr} with $r$ replaced by $r^*$.

\begin{fact}[{\cite[Example~13.7]{Livre1}}]
\label{fact1}
Let $\F$ be a reflexive, strictly convex, smooth, and real Banach space, let $r \in \left]1,+\infty\right[$,
and let $\varphi \in \Gamma_0(\R)$ be even. Then
$(\varphi \circ \norm{\cdot})^* = \varphi^* \circ \norm{\cdot}$ and
\begin{equation*}
(\forall\, w \in \F)\qquad
\partial (\varphi \circ \norm{\cdot})(w) =
\begin{cases}
\dfrac{\partial \varphi(\norm{w})}{\norm{w}^{r-1}}J_r(w)&\text{if } w \neq 0\\[2ex]
\{w^* \in \F^*~\vert~\norm{w^*} \in \partial \varphi(0)\}&\text{if } w = 0.
\end{cases}
\end{equation*}
\end{fact}

\begin{fact}[{Fenchel-Rockafellar duality \cite[Corollary 2.8.5 and Theorem~2.8.3(vi)]{Zali02}}]
\label{fact:FRduality}
Let $\F$ and $\B$ be two real Banach spaces. Let $F \in \Gamma_0(\F)$,
let $\Psi \in \Gamma_0(\B)$, and let $B\colon \F \to \B$
be a bounded linear operator. Suppose that $0 \in \mathrm{int}\big(B(\dom F) - \dom \Psi\big)$.
Then the dual problem
\begin{equation}
\label{eq:20160302c}
\min_{y^* \in \B^*} F^*(- B^* y^*) + \Psi^*(y^*)
\end{equation}
admits solutions and strong duality holds, that is
\begin{equation*}
\inf_{x \in \F} F(x) + \Psi(B x) = - \min_{y^* \in \B^*} F^*(- B^* y^*) + \Psi^*(y^*).
\end{equation*}
Moreover, %if in addition $F + \Psi\circ B$ admits a minimizer, then,
for every $(\bar{x}, \bar{y}^* ) \in \F \times \B^*$, 
$\bar{x}$ is a minimizer for $F + \Psi\circ B$ and $\bar{y}^*$
is a solution of \eqref{eq:20160302c}  if and only if $-B^* \bar{y}^* \in \partial F(\bar{x})$ and 
$\bar{y}^* \in \partial \Psi (B \bar{x})$.
\end{fact}

\begin{fact}[{Generalized Cauchy-Schwarz inequality \cite[Corollary~2.11.5]{Bog07}}]
\label{p:20150525b}
Let $\KK$ be a nonempty set. Let $m \in \N$ and let $a_1, a_2, \dots, a_m \in \ell_+^m(\KK)$.
Then $a_1 a_2 \cdots a_m \in \ell_+^1(\KK)$ and
\begin{equation*}
\sum_{k \in \KK} a_1[k] a_2[k] \cdots a_m[k] \leq \bigg(\sum_{k \in \KK} a_1[k]^m\bigg)^{1/m}
\bigg(\sum_{k \in \KK} a_2[k]^m\bigg)^{1/m} \cdots \bigg(\sum_{k \in \KK} a_m[k]^m\bigg)^{1/m}.
\end{equation*}
%where we used the following compact notation for the component-wise product
%of two sequences:
%\begin{equation*}
%(\forall\, a\in \ell^r(\KK))(\forall\, b\in \ell^{r^*}\!(\KK))
%\qquad \sum_{k \in \KK} a b : = \sum_{k \in \KK} a[k]b[k].
%\end{equation*}
\end{fact}
%\begin{proof}
%We prove it by induction. The statement is true for $m=2$.
%Suppose that the statement holds for $m \geq 2$ and let 
%$a_1, a_2, \dots, a_m, a_{m+1} \in \ell_+^{m+1}(\KK)$. Then
%$a_1^{(m+1)/m}, a_2^{(m+1)/m}, \dots, a_m^{(m+1)/m} \in \ell_+^m(\KK)$ 
%and by induction hypothesis $(a_1 a_2 \cdots a_m)^{(m+1)/m} \in \ell_+^1(\KK)$ and
%\begin{equation*}
%\sum_{k \in \KK} (a_1 a_2 \cdots a_m)^{(m+1)/m} \leq \bigg( \sum_{k \in \KK} a_1^{m+1}\bigg)^{1/m}
%\bigg( \sum_{k \in \KK}a_2^{m+1}\bigg)^{1/m} \cdots \bigg(\sum_{k \in \KK} a_m^{m+1}\bigg)^{1/m}.
%\end{equation*}
%Now, since $a_1 a_2 \cdots a_m \in \ell_+^{(m+1)/m}(\KK)$, $a_{m+1} \in \ell_+^{m+1}(\KK)$, and
%$(m+1)/m$ and $m+1$ are conjugate exponents, it follows from H\"older inequality that
%$a_1 a_2 \cdots a_m a_{m+1} \in \ell_+^1(\KK)$ and
%\begin{align*}
%\sum_{k \in \KK} a_1 a_2 \cdots a_m a_{m+1} &\leq 
%\bigg( \sum_{k \in \KK} (a_1 a_2 \cdots a_m)^{(m+1)/m} \bigg)^{m/(m+1)}
%\bigg( \sum_{k \in \KK} a_{m+1}^{m+1} \bigg)^{1/(m+1)}\\
%&\leq \bigg( \sum_{k \in \KK} a_1^{m+1}\bigg)^{1/{m+1}}
%\bigg( \sum_{k \in \KK}a_2^{m+1}\bigg)^{1/{m+1}} \cdots 
%\bigg(\sum_{k \in \KK} a_{m+1}^{m+1}\bigg)^{1/{m+1}}.
%\end{align*}
%\end{proof}

\section{General SVR
in Banach spaces of features.}
\label{sec:GenSVR}

Support vector regression aims at learning a nonlinear relation between 
an input space $\XC$ and an output space $\YC\subset \R$ based on a given
set of input-output pairs $(x_i,y_i)_{1 \leq i \leq n} \in (\XC\times\YC)^n$,  
called the \emph{training set}. The input-output relation is sought in a
\emph{hypothesis function space} of the following form
\begin{equation*}
\Big\{ f\colon \XC \to \R \,\big\vert\, (\exists(w,b)\in \HH\times\R)(\forall\, x \in \XC)\ f(x) = \pair{w}{\Phi(x)} + b \Big\},
\end{equation*}
%\begin{equation*}
%f(x) = \pair{w}{\Phi(x)}
%\end{equation*}
where $\Phi\colon \XC \to \HH$ is a nonlinear map (the \emph{feature map}) from the input space
to a Hilbert space (the \emph{feature space}). Then, support vector regression is formulated
as the following minimization problem
\begin{equation}
\label{eq:20170221a}
\min_{(w,b) \in \HH\times\R} \frac{\gamma}{n} \sum_{i=1}^n 
L_{\varepsilon}\big(y_i - \pair{w}{\Phi(x_i)} - b\big) + \frac 1 2 \norm{w}^2,
\end{equation}
where, 
$L_\varepsilon(t) = \max\{0, \abs{t} - \varepsilon\}$ is the
Vapnik's \emph{$\varepsilon$-insensitive loss} \cite{Vap98}
and $\gamma>0$ is the regularization parameter. The optimization problem
\eqref{eq:20170221a} has the drawback that often it has to be solved in an high or even infinite dimensional Hilbert space.
In such case, it may be more convenient to approach its dual problem
(see \cite[Proposition~6.21]{Cri00} and \cite{Gir98,SteChi2009,Vap98})
%It is well known
%(see \cite[Proposition~6.21]{Cri00} and \cite{Gir98,SteChi2009,Vap98}) that
%problem \eqref{eq:20170221a} can be approached through its dual problem
 \begin{equation}
\label{eq:20170221b}
\left[
\begin{aligned}
&\min_{ u \in \R^n} 
\frac{1}{2 n^{2}}  
 \sum_{i, j= 1 }^n
 K(x_{i}, x_{j}) 
 u_{i} u_{j} 
 - \frac{1}{n} \sum_{i=1}^n y_i u_i
 + \frac{\varepsilon}{n} \sum_{i=1}^n \abs{u_i}
 \\[1.5ex]
&\text{subject to } \sum_{i=1}^n u_i = 0\quad\text{and}\quad 
\abs{u_i} \leq \gamma\ \text{for every}\ i \in \{1,\dots,n\},
\end{aligned}
\right.
\end{equation}
where $K(x, x^\prime) = \pair{\Phi(x)}{\Phi(x^\prime)}$ is the \emph{kernel function}.
Once a solution $u$ of the dual problem \eqref{eq:20170221b} is obtained,
a solution $(w,b)$ of the primal problem\footnote{Since 
$\abs{\cdot} \leq L_\varepsilon + \varepsilon$, the objective function in 
\eqref{eq:20170221a} is coercive in $(w,b)$, hence a solution exists.} \eqref{eq:20170221a} is computed by
means of the \emph{representer theorem} \cite{Gir98,Sch01}
\begin{equation}
\label{eq:20170222d}
w = \frac 1 n \sum_{i=1}^n u_i \Phi(x_i),
\end{equation}
and by choosing $b$ so that $y_j - \pair{w}{\Phi(x_j)} - b= \sign(u_j)\varepsilon$, for any $j$ 
with  $0<\abs{u_j}<\gamma$.
Moreover, and more importantly, the regression function $f = \pair{w}{\Phi(\cdot)} + b$ 
can be evaluated at a new input point $x$
by using the kernel function only
\begin{equation}
\label{eq:20170315c}
f(x) = 
%\pair{w}{\Phi(x)} + b = 
\frac 1 n \sum_{i=1}^n u_i K(x_i,x) + b
= \frac 1 n \sum_{i=1}^n u_i\big( K(x_i,x) - K(x_i,x_j) \big) + y_j - \sign(u_j)\varepsilon.
\end{equation}
The significance of this theory stands in the fact that it provides a viable computational framework
for solving SVR, which is a nonparametric (infinite dimensional) estimation technique:
indeed given the kernel function (without knowing the feature map $\Phi$ itself) one
can formulate the dual optimization problem and evaluate the resulting regression function 
%indeed the explicit knowledge of the kernel function  
%(without knowing the feature map $\Phi$ itself)
%suffices to formulate the dual problem as well as to evaluate the
%regression function 
%at a new input point 
--- this is the \emph{kernel trick}
and constitutes the key idea of \emph{kernel methods} \cite{SteChi2008}.
We stress that, even in the case that $\HH$ is finite dimensional, going through the dual is
still convenient if $\mathrm{dim}\, \HH \gg n$.

The goal of the present work is to extend the above theory to the case 
of more general regularizers and more general hypothesis function spaces.
%so to \emph{kernelize} other nonlinear estimation techniques.
Popular estimation techniques that require regularizers
different from the square of Hilbert norms, are the \emph{lasso} 
and the \emph{bridge regression}, which can be formulated in a unifying manner as
\begin{equation}
\label{eq:20170222a}
%\min_{w \in \ell^1(\KK)} 
\min_{(w,b) \in \ell^r(\KK)\times\R} \frac{\gamma}{n} \sum_{i=1}^n 
\big(y_i - \pair{w}{\Phi(x_i)} - b\big)^2 + \frac 1 r \norm{w}_r^r \qquad(1 \leq r <2).
\end{equation}
%where now $\Phi(x) = \big(\phi_k(x)\big)_{k \in \KK}$
These techniques aim at finding the most relevant features $w_k$'s
in the representation of the regression function 
$f = \pair{w}{\Phi(\cdot)} + b =\sum_{k \in \KK} w_k \phi_k + b$,
when this representation is known to be sparse.
They are grounded on the fact that
%since it has been theoretically proved that 
$\norm{\cdot}_r^r$ preserves sparsity
for $r>1$ but close to $1$ \cite{Kol2009}.
However, even though the dual of the optimization problem \eqref{eq:20170222a} 
is in principle finite dimensional, the presence of the non-Hilbertian norm $\norm{\cdot}_r$
breaks the quadratic structure of the dual problem and
%does not allow anymore to define any useful kernel function describing
does not allow to define any useful kernel function describing 
the dual problem as well as the regression function.
Indeed the kernels defined in \cite{Com15,Zhang2009,Zhang2012},
in the setting of reproducing kernel Banach spaces, are not suitable
for that purpose (see Remark~\ref{rmk:20170315a}).
%We note that, in \cite{Com15,Zhang2009,Zhang2012} 
%kernels are indeed defined for function Banach spaces, but
%they fail to properly represent the dual problem and the estimator 
%(see Remark~\ref{rmk:20170315a}). 
In the next section we show
that the estimation technique \eqref{eq:20170222a} can be naturally \emph{kernelized}
for certain choices of $r$, provided that one enlarges the concept of kernel functions.
%($1<r<2$) breaks the representation by kernels in 
%\eqref{eq:20170221b} and \eqref{eq:20170315c} and 
%it is 
%not clear how to replace the feature map (which is an infinite dimensional object)
%with a suitable kernel function that describes the dual problem as well as the estimator $f$
%(in \cite{Com15,Zhang2009,Zhang2012} kernels are defined for function Banach spaces, however,
%they fail to properly represent the dual problem and the estimator).

In the following we consider duality for a continuous version of 
the support vector regression problem \eqref{eq:20170221a}
and for general convex regularizers and loss functions. So we address the optimization problem
\begin{equation}
\label{eq:mainprob1}
\min_{(w,b) \in \FF\times \R}\gamma \int_{\XC \times \YC} 
\loss\big(y - \pair{w}{\Phi(x)} - b\big) \ud P(x,y) + G(w),
\end{equation}
where the following assumptions are made:
\begin{description}
\item[A1] 
$\XC$ and $\YC$ are two nonempty sets such that $\YC \subset \R$.
 $\PP$ is
 a probability distribution on $\XC\times \YC$, defined on some
underlying $\sigma$-algebra $\mathfrak{A}$ on $\XC\times\YC$.
 $\FF$ is a real separable  reflexive Banach space and 
$\Phi\colon \XC \to \FF^*$ is a  measurable function.
The function $\loss\colon \R \to \RP$ is positive and convex,\!\!
\footnote{Usually one requires that $L$
is also even. In that case
it is easy to see that necessarily $0$ is a minimizer of $\loss$ 
and that $\loss$ is increasing on $\RP$. Indeed 
for every $t \in \RP$, we have $-t \leq 0 \leq t$, and hence
$0 = (1 - \alpha)(-t) + \alpha t$, for some $\alpha \in [0,1]$.
Then, by convexity $\loss(0) \leq (1 - \alpha)\loss(-t) + \alpha \loss(t) = \loss(t)$,
for $\loss(-t)=\loss(t)$. Moreover, for every $s,t \in \R$, with $0\leq s \leq t$, we have
$s = (1- \alpha)0 + \alpha t$, for some $\alpha \in [0,1]$, and hence
$\loss(s) \leq (1 - \alpha)\loss(0) + \alpha \loss(t)$ which yields
$\loss(s) - \loss(0) \leq \alpha (\loss(t) - \loss(0)) \leq \loss(t) - \loss(0)$.} 
 $p\in \left[1,+\infty\right[$, $\gamma \in \RPP$,
and $G\colon\FF \to \RX$ is proper, lower semicontinuous, 
 and convex.
 \item[A2] $(\exists\, (a,b) \in \RP^2)(\forall\, t \in \R)\quad \loss(t) \leq a + b \abs{t}^p$.
\item[A3] $\displaystyle
\int_{\XC\times\YC} \abs{y}^p \ud P(x,y)<+\infty
\quad\text{and}\quad
 \int_{\XC\times\YC} \norm{\Phi(x)}^p \ud \PP(x,y)<+\infty$.
\end{description}

In this context $\FF$ and $\Phi$ are respectively the \emph{feature space}
and the \emph{feature map}, and  $L$ is
the \emph{loss} function \cite{Com15,Zhang2009}.
%Problem \eqref{eq:mainprob1} reduces  be considered as a continuous
%version of support vector regression,
%for general loss $\loss$ and  regularizer $G$. Indeed,
If $\PP$ is chosen as a discrete distribution, say
$\PP = (1/n) \sum_{i=1}^n \delta_{(x_i,y_i)}$, for some
sample $(x_i,y_i)_{1 \leq i \leq n} \in (\XC\times\YC)^n$, then \eqref{eq:mainprob1} reduces to
the optimization problem
\begin{equation*}
\min_{(w,b) \in \FF\times \R}\frac{\gamma}{n}\sum_{i=1}^n
\loss\big(y_i - \pair{w}{\Phi(x_i)} - b \big) + G(w),
\end{equation*}
which 
%is the way support vector regression is formulated in \cite{Gir98} and that 
encompasses problems \eqref{eq:20170221a} and \eqref{eq:20170222a}.
Assumption {\bf A2} corresponds to an upper 
growth condition for the loss $\loss$, whereas
assumption {\bf A3} includes a moment condition for the distribution $\PP$
and an integrable condition for the feature map $\Phi$, with respect to $\PP$.
They are both standard assumptions in support vector machines \cite{SteChi2008}
and ensure that the integral in \eqref{eq:mainprob1} is finite for every $(w,b) \in \FF\times \R$.
In the following we consider the Lebesgue space
\begin{equation*}
L^p(\PP) = \bigg\{ u\colon\XC\times\YC \to 
\R~\Big\vert~u \text{ is $\mathfrak{A}$-measurable and }
\int_{\XC\times\YC} \abs{u(x,y)}^p d \PP(x,y)<+\infty \bigg\}.
\end{equation*}
Problem \eqref{eq:mainprob1} is a convex optimization problem
of a composite form.
%on an infinite dimensional space. 
The following result 
first recasts the problem in a constrained form,
as done in
\cite{Cri00,Suy2002}, then presents
its dual problem, with respect to the Fenchel-Rockafellar
duality, and the related optimality conditions (Fact~\ref{fact:FRduality}).

\begin{theorem}
\label{thm:pridu}
Let assumptions  {\bf A1}, {\bf A2}, and {\bf A3} hold. Then problem \eqref{eq:mainprob1}
is equivalent to
\begin{equation}
\tag{$\mathcal{P}$}
\left[
\begin{aligned}
&\min_{(w,b,e) \in \FF\times\R \times L^p(P) }
\gamma \int_{\XC\times\YC} \loss(e(x,y))\ud P(x,y) + G(w),\\[1ex]
&\text{subject to}\ \ 
y -\pair{w}{\Phi(x)} - b  = e(x,y)
\quad \text{for $\PP$-a.a.}\ (x,y) \in \XC\times\YC
\end{aligned}
\right.
\end{equation}
and its dual is
\begin{equation}
\tag{$\mathcal{D}$}
\left[
\begin{aligned}
&\min_{ u \in L^{ p^{*}}\!(\PP)} 
G^*\bigg( \int_{\XC\times\YC} u(x,y)\Phi(x) \ud P(x,y) \bigg) \\
&\hspace{16ex}+ \gamma \int_{\XC \times \YC} L^*\bigg(\frac{u(x,y)}{\gamma}\bigg) \ud \PP(x,y) 
-  \int_{\XC \times \YC} y\, u(x,y) \ud \PP(x,y)\\[2ex]
&\text{subject to } \int_{\XC\times\YC} u\ud P = 0.
\end{aligned}
\right.
\end{equation}
Moreover, the dual problem $(\mathcal{D})$ admits solutions, strong duality holds, 
and for every 
$\displaystyle (w,b,e) \in \FF\times\R \times L^p(P)$ and every 
$u \in L^{p^*}(\PP)$, we have that
$(w,b,e)$ is a solution of $(\mathcal{P})$ and $u$ is a solution of $(\mathcal{D})$ 
if and only if the following optimality conditions hold
\begin{equation}
\label{eq:optPD}
\begin{cases}
\displaystyle w \in \partial G^* \bigg( \int_{\XC\times\YC} u(x,y)\Phi(x) \ud\PP(x,y) \bigg)\\[2ex]
\displaystyle \int_{\XC\times\YC} u \ud \PP = 0\\[3ex]
\dfrac{u(x,y)}{\gamma} \in \partial \loss(e(x,y))
\quad \text{for $\PP$-a.a.}\ (x,y) \in \XC\times\YC\\[1.5ex]
y - \pair{w}{\Phi(x)} - b = e(x,y)
\quad \text{for $\PP$-a.a.}\ (x,y) \in \XC\times\YC.
\end{cases}
\end{equation}
\end{theorem}

\begin{remark}\ 
\begin{enumerate}[$(i)$]
\item\label{it:20170315a} If $L$ and $G$ are coercive, that is $\lim_{\abs{t}\to +\infty} L(t) = +\infty$
and $\lim_{\norm{w}\to +\infty} G(w)=+\infty$ (e.g., this is the case of \eqref{eq:20170221a}), 
then the primal problem $(\mathcal{P})$
admits solutions. Moreover, if in addition $L$ and $G$ are strictly convex
(as for \eqref{eq:20170222a}), the solution is unique.
\item If the offset $b=0$, condition $\int_{\XC\times\YC} u \ud \PP = 0$ in \eqref{eq:optPD} 
can be omitted. Moreover, the coercivity (resp.~the strictly convexity)
of $G$ only suffices to get the existence (resp.~uniqueness) of solutions of
problem $(\mathcal{P})$.
%discussed in \ref{it:20170315a}.
\end{enumerate}
\end{remark}

\begin{remark}\ 
\begin{enumerate}[$(i)$]
\item The form $(\mathcal{P})$ resembles the way the problem of 
support
vector machines for regression is often formulated \cite[eq. (3.51)]{Suy2002} and
the optimality conditions \eqref{eq:optPD} are the continuous
versions of the one stated in \cite[eq. (3.52)]{Suy2002} for RKHS,
 differentiable loss functions, and square norm regularizers.
 \item Let $u$ be a solution of the dual problem $(\mathcal{D})$.
Then the (possibly empty) set of solutions of the primal problem $(\mathcal{P})$
 is composed by the $(w,b,e)$'s satisfying the optimality conditions \eqref{eq:optPD}.
Equivalently, recalling \eqref{eq:20160302a}, the solutions of the primal problem, 
in the form \eqref{eq:mainprob1}, are given by
\begin{equation}
\label{eq:optPD2}
\begin{cases}
\displaystyle w \in \partial G^* \bigg( \int_{\XC\times\YC} u(x,y)\Phi(x) \ud\PP(x,y) \bigg)\\[2ex]
 b  \in y - \pair{w}{\Phi(x)} - \partial \loss^*\bigg(\dfrac{u(x,y)}{\gamma}\bigg)
\quad \text{for $\PP$-a.a.}\ (x,y) \in \XC\times\YC.
\end{cases}
\end{equation}
\item If $G$ is strictly convex on every convex subset of $\dom \partial G$
and $\mathrm{int}(\dom G^*) = \dom \partial G^*$, then $G^*$ is 
G\^ateaux differentiable (hence $\partial G^*$ is single valued) 
on $\dom \partial G^*$ \cite[Proposition~18.9]{Livre1} and,
if $(w,b)$ is  a solution of the primal problem \eqref{eq:mainprob1}, then
the first of \eqref{eq:optPD2} yields
\begin{equation}
\label{eq:rep}
w = \nabla G^* \bigg( \int_{\XC\times\YC} u\Phi \ud\PP \bigg).
\end{equation}
This constitutes a general \emph{nonlinear representer theorem}, since the solution
of problem $(\mathcal{P})$ is expressed in terms of the values of the feature map $\Phi$.
When $\PP$
is the discrete distribution $\PP = (1/n) \sum_{i=1}^n \delta_{(x_i,y_i)}$, for some
sample $(x_i,y_i)_{1 \leq i \leq n} \in (\XC\times\YC)^n$, then 
\eqref{eq:rep} becomes
\begin{equation}
\label{eq:repd}
w = \nabla G^* \bigg( \frac 1 n \sum_{i=1}^n  u_i \Phi(x_i) \bigg).
\end{equation}
%We note that in \eqref{eq:rep}-\eqref{eq:repd}
% the nonlinearity relies on the mapping $\nabla G^*$ only.
 \item In the special case that $\FF$ is a Hilbert space,
$\FF$ is isomorphic to its dual and the pairing 
reduces to the inner product in $\FF$. If $b=0$ and
$G = (1/2) \norm{\cdot}^2$, then
$\nabla G^* = \mathrm{Id}$ and
the optimality conditions \eqref{eq:optPD} reduce to the equations
\begin{equation*}
w = \int_{\XC\times\YC} u\Phi \ud\PP
\quad\text{and}\quad\frac{u(x,y)}{\gamma} \in \partial L(y - \pair{w}{\Phi(x)})
\quad \text{for $\PP$-a.a.}\ (x,y) \in \XC\times\YC,
\end{equation*}
which were obtained in  \cite[Corollary~3]{DeVito04}.
If additionally, $\PP$ is discrete, then \eqref{eq:repd} turns to \eqref{eq:20170222d}.
\item If $\FF = \ell^2(\KK)$ and $G = \norm{\cdot}_1 +  1/(2\tau) \norm{\cdot}^2$ (elastic net regularization),
then $G^*$ is strongly convex, $\nabla G^* = \prox_{\tau\norm{\cdot}_1} (\tau\,\cdot)$ \cite{Livre1}, 
and \eqref{eq:rep} and \eqref{eq:repd} turn respectively 
to the following representation formulas
\begin{equation*}
w = \prox_{\tau\norm{\cdot}_1}\bigg(\tau \int_{\XC\times\YC} u\Phi \ud\PP\bigg)
\quad\text{and}\quad
w = \prox_{\tau\norm{\cdot}_1}\bigg(\frac{\tau}{n}  \sum_{i=1}^n u_i \Phi(x_i)\bigg),
\end{equation*}
where $\prox_{\tau\norm{\cdot}_1}$ acts component-wise as a soft-thresholding operation with threshold $\tau$.
\end{enumerate}
\end{remark}

The optimality conditions \eqref{eq:optPD} in Theorem~\ref{thm:pridu}
yield the following continuous representer theorem
in Banach space setting and for regularizers that are function of the norm.

\begin{corollary}[Continuous representer theorem]
\label{cor:20160223m}
Let assumptions  {\bf A1}, {\bf A2}, and {\bf A3} hold.
Suppose that $\FF$ is strictly
convex and smooth and let $r \in \left]1,+\infty\right[$. 
In problem $(\mathcal{P})$, suppose that $G = \varphi \circ \norm{\cdot}$, for some 
convex and even function $\varphi\colon \R \to \R_+$
such that $\argmin \varphi = \{0\}$.
Let $(w,b)$ be a solution of problem $(\mathcal{P})$. Then $w$ admits the following representation
\begin{equation}
\label{eq:20160223l}
J_r(w) =  \int_{\XC\times\YC} c(x,y)\Phi(x) \ud \PP(x,y),
\end{equation}
for some function $\displaystyle c \in L^{p^*}(\PP)$,
where $J_r\colon \FF \to \FF^*$ is the $r$-duality map of $\FF$.
\end{corollary}

\begin{remark}
\label{rmk:20170315a}
 If in Corollary~\ref{cor:20160223m}, $r=2$ and $\PP$ is a discrete measure, say
 $\PP = (1/n)\sum_{i=1}^n\delta_{(x_i,y_i)}$, for some
sample $(x_i,y_i)_{1 \leq i \leq n} \in (\XC\times\YC)^n$, then
\eqref{eq:20160223l} becomes
\begin{equation}
\label{eq:20160229a}
J_2(w) = \sum_{i=1}^n c_i \Phi(x_i),\quad {(c_i)}_{1 \leq i \leq n} \in \R^n,
\end{equation}
where $J_2$ is the normalized duality map of $\FF$.
Formula \eqref{eq:20160229a} is the way the representer theorem is usually
presented in reproducing kernel Banach spaces \cite{Fass15,Zhang2009,Zhang2012}. Here it
is a simple consequence of the more general 
Theorem~\ref{thm:pridu} and Corollary~\ref{cor:20160223m}.
Moreover, we stress that our derivation of \eqref{eq:20160229a} relies
on convex analysis arguments only, while in the above cited literature
it is proved as a consequence of a representer theorem for
function interpolation, ultimately using different techniques and stronger hypotheses.
We note also that in Banach space setting \cite{Zhang2009,Zhang2012}, the kernel is defined as
\begin{equation*}
K(x,x^\prime) = \pair{J_2^{-1}(\Phi(x))}{\Phi(x^\prime)},
\end{equation*}
but this kernel function is inadequate for describing the dual problem 
and evaluating the regression function $\pair{w}{\Phi(x)} + b$ at a new point $x$ \cite{Sri11}.
%We finally note that, if
% $\FF$ is a Hilbert space and $r=2$, then $J_2$ is the identity map
% of $\FF$ and \eqref{eq:20160229a} becomes
% $w =  \sum_{i=1}^n c_i\Phi(x_i)$, which is the 
% classical representer theorem in Hilbert spaces \cite{Sch01}.
\end{remark}

\begin{example}
\label{ex:epsinsensitive}
We consider the case of the  \emph{$\varepsilon$-insensitive loss} \cite{SteChi2009,Vap98}.
Let $\varepsilon>0$ and define
\begin{equation}
\label{eq:epsloss}
\loss_{\varepsilon} \colon \R \to \R_+\colon t \mapsto \max\{0, \abs{t} - \varepsilon\}.
\end{equation}
This loss clearly satisfies {\bf A2} for every $p\geq 1$.
We note that \eqref{eq:epsloss} is 
the distance function from the set $[-\varepsilon,\varepsilon]$,
that is, using the notation in \cite{Livre2}, we have $\loss_{\varepsilon} 
= d_{[-\varepsilon,\varepsilon]}$. Then, the Fenchel conjugate of $\loss_{\varepsilon}$
is (see \cite[Example~13.24(i)]{Livre2})
\begin{equation*}
\loss^*_{\varepsilon} = \sigma_{[-\varepsilon,\varepsilon]} + \iota_{[-1,1]} =
\varepsilon \abs{\cdot} + \iota_{[-1,1]}.
\end{equation*}
Therefore, for the loss \eqref{eq:epsloss}, the dual problem $(\mathcal{D})$ becomes
\begin{equation*}
\left[
\begin{aligned}
&\min_{ u \in L^{ p^{*}}\!(\PP)} 
G^*\bigg( \int_{\XC\times\YC} u(x,y)\Phi(x) \ud P(x,y) \bigg) \\
&\hspace{16ex}+ \varepsilon \int_{\XC \times \YC} \abs{u(x,y)} \ud \PP(x,y) 
- \int_{\XC \times \YC} y\, u(x,y) \ud \PP(x,y)\\[2ex]
&\text{subject to } \int_{\XC\times\YC} u\ud P = 0\quad\text{and}\quad
\abs{u(x,y)}\leq \gamma
\ \text{for $\PP$-a.a.}\ (x,y) \in \XC\times\YC.
\end{aligned}
\right.
\end{equation*}
This is a continuous version of the dual problem \eqref{eq:20170221b}, 
%that arises in classical  support vector regression when the linear $\varepsilon$-insensitive loss
%is considered 
where here we have a general regularizer and
a Banach feature space.
\end{example}

\begin{remark}\
Let us consider the case that $\FF$ is a Hilbert space. 
Moreover,  suppose that $G = (1/2) \norm{\cdot}^2$,
 that $L= (1/2)\abs{\cdot}^2$, and that $b=0$, so that
 in \eqref{eq:optPD} the condition $\int_{\XC\times\YC} u \ud \PP = 0$
can be neglected.
Then it follows from  \eqref{eq:optPD} that
\begin{equation*}
w=\int_{\XC\times\YC} u \Phi \ud \PP, \quad \frac{u}{\gamma} = e
\end{equation*}
and hence
\begin{equation*}
\pair{w}{\Phi(x)} =
\int_{\XC\times\YC} u(x^\prime,y^\prime) \pair{\Phi(x^\prime)}{\Phi(x)} \ud \PP(x^\prime,y^\prime).
\end{equation*}
Thus, the last of \eqref{eq:optPD} yields the following integral equation
\begin{equation*}
(\forall\, (x,y) \in \XC\times\YC)\qquad \frac{u(x,y)}{\gamma} + 
\int_{\XC\times\YC} u(x^\prime,y^\prime) \pair{\Phi(x^\prime)}{\Phi(x)} \ud \PP(x^\prime,y^\prime)
= y.
\end{equation*}
\end{remark}

\section{Tensor-kernel representation}
\label{sec:kernrep}

In this section we study a framework that includes SVR problems of type \eqref{eq:20170222a}
(for certain choices of $r$) and that provides a new tensorial \emph{kernelization} of the
dual problem $(\mathcal{D})$.
%We present our framework. 
For clarity we consider separately the
real and complex case. We describe the real case with full details,
whereas in the complex case we provide results with sketched
 proofs only.

\subsection{The real case}
\label{subsec:real}
Let $\displaystyle\FF = \ell^r(\KK)$, with $\KK$ a countable set and 
$r = m/(m-1)$ for some even integer  $m\geq 2$. Thus, we have $r^* = m$.
Let $(\phi_k)_{k \in \KK}$ be a family of measurable functions from $\XC$ to $\R$ such that,
for every $x \in \XC$, $\displaystyle {(\phi_k(x))}_{k \in \KK} \in \ell^{r^*}\!(\KK)$
and define the feature map as
\begin{equation}
\label{eq:20160225a}
\Phi\colon \XC \to \ell^{r^*}\!(\KK) 
\colon x \mapsto (\phi_k(x))_{k \in \KK}.
\end{equation}
Thus, we consider the following  linear model
\begin{equation}
\label{eq:20160225b}
\big(\forall\, (w,b) \in \ell^r(\KK)\times\R \big)\qquad f_{w,b} =\pair{w}{\Phi(\cdot)}_{r,r^*} + b
=  \sum_{k \in \KK} w_k \phi_k + b\ \text{(pointwise)},
\end{equation}
where $\pair{\cdot}{\cdot}_{r,r^*}$ is the canonical pairing between $\ell^r(\KK)$
and $\ell^{r^*}\!(\KK)$.
The space
\begin{equation}
\label{eq:20160305a}
\bb = \bigg\{ f\colon\XC \to \R~\Big\vert~(\exists (w,b) \in \ell^r(\KK)\times\R)(\forall\, x \in \XC)
\bigg(f(x) = \sum_{k \in \KK} w_k\phi_k(x) + b\bigg)\bigg\}
\end{equation}
is a \emph{reproducing kernel Banach space} with norm
\begin{equation*}
(\forall\, f \in \bb)\quad
\norm{f}_\bb = \inf \bigg\{ \norm{w}_r + \abs{b}~\Big\vert~(w,b)\in \ell^r(\KK)\times\R\ \text{and}\ 
f = \sum_{k \in \KK} w_k \phi_k+ b\ \text{(pointwise)} \bigg\},
\end{equation*}
meaning that, with respect to that norm, 
the point-evaluation operators are continuous \cite{Com15,Zhang2009}.
We also consider the following regularization function
\begin{equation}
\label{eq:20160225c}
G(w) = \varphi(\norm{w}_r),
\end{equation}
for some convex and even function $\varphi\colon \R \to \R_+$,
such that $\argmin \varphi = \{0\}$,
and we set 
\begin{equation}
\label{eq:20170309a}
P =\frac 1 n \sum_{i=1}^n \delta_{(x_i,y_i)},
\end{equation}
for some given sample ${(x_i,y_i)}_{1 \leq i \leq n} \in (\XC \times \YC)^n$. 

\begin{remark}
\label{rmk:schauder}
Consider the reproducing kernel Banach space
\begin{equation*}
\bb = \bigg\{ f\colon\XC \to \R~\Big\vert~(\exists w \in \ell^r(\KK))(\forall\, x \in \XC)
\bigg(f(x) = \sum_{k \in \KK} w_k\phi_k(x) \bigg)\bigg\}
\end{equation*}
endowed with norm
$\norm{f}_\bb = \inf \big\{ \norm{w}_r~\big\vert~w\in \ell^r(\KK)\ \text{and}\ 
f = \sum_{k \in \KK} w_k \phi_k\ \text{(pointwise)} \big\}$.
Let $f \in \bb$ and let $\displaystyle{(w_k)}_{k \in \KK} \in \ell^r(\KK)$ be
such that $f = \sum_{k \in \KK} w_k \phi_k$ pointwise.
Then, for every finite subset $\mathbb{J} \subset \KK$ we have
$f - \sum_{k \in \mathbb{J}} w_k \phi_k= \sum_{k \in \KK\setminus \mathbb{J}} w_k \phi_k$
pointwise; hence, by definition 
\begin{equation*}
\bigg\lVert f - \sum_{k \in \mathbb{J}} w_k \phi_k \bigg\rVert_\bb 
\leq \big\lVert{(w_k)}_{k \in \KK\setminus \mathbb{J}}\big\rVert_r 
= \bigg( \sum_{k \in \KK\setminus \mathbb{J} } \abs{w_k}^r \bigg)^{1/r} \to 0
\quad\text{as}\ \abs{\mathbb{J}}\to +\infty.
\end{equation*}
Thus, the family ${(w_k \phi_k)}_{k \in \KK}$ is summable in $(\bb, \norm{\cdot}_\bb)$
and it holds $f = \sum_{k \in \KK} w_k \phi_k$ in $(\bb, \norm{\cdot}_\bb)$.
Therefore, if the family of functions ${(\phi_k)}_{k \in \KK}$ is pointwise
$\ell^r$-independent, in the sense that
\begin{equation}
\label{eq:lrind}
(\forall\, {(w_k)}_{k \in \KK} \in \ell^r(\KK))\quad \sum_{k \in \KK} w_k \phi_k = 0
\ \text{(pointwise)}\ \Rightarrow {(w_k)}_{k \in \KK} \equiv 0,
\end{equation}
then $(\phi_k)_{k \in \KK}$ is an unconditional Schauder basis of
$\bb$. Indeed if $\sum_{k \in \KK} w_k \phi_k = 0$
in $(\bb, \norm{\cdot}_\bb)$, since the evaluation operators on $\bb$ are continuous, we have
$\sum_{k \in \KK} w_k \phi_k = 0$ pointwise, and hence, by \eqref{eq:lrind}, 
${(w_k)}_{k \in \KK} \equiv 0$. We finally note that when
$(\phi_k)_{k \in \KK}$ is a (unconditional) Schauder basis of $\bb$,
then $\bb$ is isometrically isomorphic to $\ell^r(\KK)$.
\end{remark}

In the setting \eqref{eq:20160225a}--\eqref{eq:20170309a}, 
the primal and dual problems 
considered in Theorem~\ref{thm:pridu} turn into
\begin{equation}
\tag{$\mathcal{P}_n$}
\left[
\begin{aligned}
&\min_{(w,b,e) \in \ell^r(\KK)\times\R \times \R^n }
\frac{\gamma}{n} \sum_{i=1}^n  \loss(e_i) + \varphi(\norm{w}_r),\\[1ex]
&\text{subject to}\ \ 
y_i -\pair{w}{\Phi(x_i)}_{r,r^*} - b  = e_i,
\quad \text{for every}\ i \in \{1,\dots, n\}
\end{aligned}
\right.
\end{equation}
and, since $G^* = \varphi^* \circ \norm{\cdot}_{r^*}$
(Fact~\ref{fact1}),
\begin{equation}
\tag{$\mathcal{D}_n$}
\left[
\begin{aligned}
&\min_{ u \in \R^n} 
\varphi^*\bigg(\bigg\lVert \frac{1}{n}\sum_{i=1}^n u_i\Phi(x_i) \bigg\rVert_{r^*} \bigg) 
+ \frac{\gamma}{n} \sum_{i=1}^n L^*\bigg(\frac{u_i}{\gamma}\bigg) 
- \frac{1}{n} \sum_{i=1}^n y_i u_i\\[1.5ex]
&\text{subject to } \sum_{i=1}^n u_i = 0.
\end{aligned}
\right.
\end{equation}
Moreover, 
%assuming that $(\mathcal{P}_n)$ admits a solution $(w,b)$, 
Fact~\ref{fact1} and \eqref{eq:optPD} yield that
$(w,b,e)$ solves $(\mathcal{P}_n)$ and $u$ solves $(\mathcal{D}_n)$
if and only if
\begin{equation}
\label{eq:optPDn}
\begin{cases}
\displaystyle w \in 
\frac{\partial \varphi^*\big( \frac{1}{n} \big\lVert \sum_{i=1}^n u_i \Phi(x_i)\big\rVert_{r^*} \big)}
{\big\lVert  \sum_{i=1}^n u_i \Phi(x_i)\big\rVert_{r^*}^{r^* -1}} 
J_{r^*} \bigg( \sum_{i=1}^n u_i\Phi(x_i) \bigg)\\[2ex]
\displaystyle \sum_{i=1}^n u_i = 0\\[3ex]
u_i/\gamma \in \partial \loss(e_i)
\quad \text{for every}\ i \in \{1,\dots,n\}\\[1.5ex]
y_i - \pair{w}{\Phi(x_i)}_{r,r^*} - b = e_i
\quad \text{for every}\ i \in \{1,\dots, n\},
\end{cases}
\end{equation}
where 
$J_{r^*}\colon \ell^{r^*}\!(\KK) \to \ell^{r}(\KK)\colon 
u \mapsto {\big(\abs{u_k}^{r^* - 1} \sign(u_k) \big)}_{k \in \N}$ and
we assumed that in the first equation of \eqref{eq:optPDn},
the right hand side has to be interpreted as $\{0\}$ when $\sum_{i=1}^n u_i \Phi(x_i)=0$.\!\!
\footnote{Since  
$G^* = \varphi^* \circ \norm{\cdot}_{r^*}$
and $\{0\} = \argmin \varphi = \partial \varphi^*(0)$,
Fact~\ref{fact1} yields
$\partial (\varphi^* \circ \norm{\cdot}_{r^*})(w^*) = 
\frac{\partial \varphi^* (\norm{w^*}_{r^*})}{\norm{w^*}_{r^*}^{r^*-1}} J_{r^*}(w^*)$ if 
$w^*\neq 0$, and $\partial (\varphi^* \circ \norm{\cdot}_{r^*})(w^*) = \{0\}$ if $w^*=0$.
}

The dual problem $(\mathcal{D}_n)$ is a
convex optimization problem and it is finite dimensional,
since it is defined on $\R^n$. Once $(\mathcal{D}_n)$ is solved, expressions
in \eqref{eq:optPDn} in principle allow
to recover the primal solution $(w,b)$ and eventually to compute the estimated
regression  function
$\pair{w}{\Phi(x)}+b$ at a generic point $x$ of the input space $\XC$.
However, if $\KK$ is an infinite set, that procedure is not feasible 
in practice, since it relies on
the explicit knowledge of
the feature map $\Phi$
%which is an infinite dimensional object,
and on the computation of an infinite dimensional scalar product.
In the following we show that, in the dual problem $(\mathcal{D}_n)$, 
we can actually get rid of the feature map $\Phi$
and use instead a new type of kernel function
evaluated at the sample points $(x_i)_{1 \leq i \leq n}$.
This will ultimately provide a new and effective computational framework
for treating support vector regression in Banach spaces of type \eqref{eq:20160305a}.

%We start by first providing a generalized Cauchy-Schwarz inequality
%for sequences which is a consequence of
%a standard generalization of H\"older's inequality 
%\cite[Corollary~2.11.5]{Bog07} and that we prove 
%for completeness. We use the following compact notation for the component-wise product
%of two sequences:
%\begin{equation*}
%(\forall\, a\in \ell^r(\KK))(\forall\, b\in \ell^{r^*}\!(\KK))
%\qquad \sum_{k \in \KK} a b : = \sum_{k \in \KK} a[k]b[k].
%\end{equation*}

Now we are ready to define a tensor-kernel associated to the feature map \eqref{eq:20160225a} 
and give its main properties.

\begin{proposition}
\label{p:tensor}
In the setting \eqref{eq:20160225a} described above,
with $r^*=m$ even integer,
the following function is well-defined
\begin{equation}\label{eq:20150525b}
K\colon \XC^m= \underbrace{\XC\times \dots\times \XC}_{m\text{ times}} \to \R\colon (x^\prime_1, \dots, x^\prime_m) 
\mapsto \sum_{k \in \KK}  \phi_k(x^\prime_{1})\cdots \phi_k(x^\prime_{m}),
\end{equation}
and the following hold.
\begin{enumerate}[$(i)$]
\item\label{p:tensorii} 
For every 
$(x^\prime_1, \dots, x^\prime_m) \in \XC^m$, and for every permutation $\sigma$
of the indexes $\{1,\dots, m\}$,
\begin{equation*}
K(x^\prime_{\sigma(1)} \dots x^\prime_{\sigma(m)}) = 
K(x^\prime_{1}, \dots x^\prime_{m}).
\end{equation*}
\item\label{p:tensorv}  
For every ${(x_i)}_{1 \leq i \leq n}\in \XC^n$
\begin{equation*}
(\forall\, u \in \R^n)\qquad\sum_{i_1,\dots, i_m = 1}^n K(x_{i_1}, \dots, x_{i_m}) u_{i_1} \dots u_{i_m} \geq 0\,.
\end{equation*}
\item\label{p:tensoriv} 
For every ${(x_i)}_{1 \leq i \leq n}\in \XC^n$
\begin{equation}\label{eq:20150527a}
u \in \R^n \mapsto \bigg\lVert \sum_{i=1}^n u_i \Phi(x_i) \bigg\rVert^{r^*}_{r^*}
= \sum_{i_1,\dots, i_m = 1}^n K(x_{i_1}, \dots, x_{i_m}) u_{i_1} \dots u_{i_m}
\end{equation}
is a homogeneous polynomial form of degree $m$ on $\R^n$.
\item\label{p:tensori} 
For every $x \in \XC$, $K(x, \dots, x) \geq 0$.
\item\label{p:tensoriii} 
For every $(x^\prime_1, \dots, x^\prime_m) \in \XC^m$
\begin{equation*}
\abs{K(x^\prime_1, \dots, x^\prime_m)} \leq K(x^\prime_1, \dots, x^\prime_1)^{1/m} 
\cdots K(x^\prime_m, \dots, x^\prime_m)^{1/m}.
\end{equation*}
\end{enumerate}
\end{proposition}
\begin{proof}
Since $(\phi_k(x^\prime_1))_{k \in \KK}, (\phi_k(x^\prime_2))_{k \in \KK}, \dots
(\phi_k(x^\prime_m))_{k \in \KK} \in \ell^m(\KK)$, it follows from 
Fact~\ref{p:20150525b} 
that $(\phi_k(x^\prime_1) \phi_k(x^\prime_2)
\cdots \phi_k(x^\prime_m))_{k \in \KK} \in \ell^1(\KK)$ and
\begin{equation}\label{eq:20150526c}
\sum_{k \in \KK}\abs{\phi_k(x^\prime_1)\cdots \phi_k(x^\prime_m)} \leq
\bigg(\sum_{k \in \KK} \abs{\phi_k(x^\prime_1)}^m \bigg)^{1/m}\!\!\!\!\!\! \cdots
\bigg(\sum_{k \in \KK} \abs{\phi_k(x^\prime_m)}^m \bigg)^{1/m}.
\end{equation}
This shows that definition \eqref{eq:20150525b} is well-posed
and moreover, since $m$ is even we can remove the absolute values in the
right hand side of \eqref{eq:20150526c} and get \ref{p:tensoriii}.
Properties \ref{p:tensorii} and \ref{p:tensori} are immediate
from the definition of $K$.
Finally, since $r^* = m$ is even, 
for every $u \in \R^n$, we have
\begin{align}
\label{eq:20160209a}
\nonumber\bigg\lVert \sum_{i=1}^n u_i \Phi(x_i) \bigg\rVert^{r^*}_{r^*} 
\nonumber&= \sum_{k \in \KK} \bigg( \sum_{i=1}^n u_i \phi_k(x_i) \bigg)^{m}\\
\nonumber&= \sum_{k \in \KK} \sum_{i_1,\dots, i_m = 1}^n \phi_k(x_{i_1})\cdots \phi_k(x_{i_m}) u_{i_1} \dots u_{i_m} \\
&=\sum_{i_1,\dots, i_m = 1}^n \bigg(\sum_{k \in \KK}  \phi_k(x_{i_1})\cdots \phi_k(x_{i_m}) \bigg) u_{i_1} \dots u_{i_m}\,.
\end{align}
Therefore, recalling the definition of $K$,  \ref{p:tensorv} and \ref{p:tensoriv} follow.
\end{proof}

\begin{remark}
Let ${(x_i)}_{1 \leq i \leq n}\in \XC^n$. Then
$\KKK:= ( K(x_{i_1}, \dots, x_{i_m}))_{i \in \{1,\dots n\}^m}$ defines a 
tensor of degree $m$ on $\R^n$. Then,
properties \ref{p:tensorii} and \ref{p:tensorv} 
 establish that the tensor is symmetric and positive definite:
they are natural generalization of 
the defining properties of standard positive (matrix) kernels.
\end{remark}

Because of Proposition~\ref{p:tensor}\ref{p:tensoriii},
tensor kernels, as defined in \eqref{eq:20150525b}, can be normalized
as for the matrix kernels.

\begin{proposition}[normalized tensor kernel]
Let $K$ be defined as in \eqref{eq:20150525b} and suppose that,
for every $x \in \XC$, $K(x, \dots, x)>0$. Define
\begin{equation}
\begin{aligned}
\tilde{K}\colon \XC^m\ &\to \R, \\
(x^\prime_{1}, \dots, x^\prime_{m}) &\mapsto
\frac{K(x^\prime_{1}, \dots, x^\prime_{m})}
{{K(x^\prime_{1}, \dots, x^\prime_{1})}^{1/m} \cdots 
{K(x^{\prime}_{m}, \dots, x^{\prime}_{m})}^{1/m}}.
\end{aligned}
\end{equation}
Then $\tilde{K}$ is still of type \eqref{eq:20150525b}, for some
family of functions  ${(\tilde{\phi}_k)}_{k \in \KK}$, $\tilde{\phi}_k\colon \XC \to \R$,
and the following hold.
\begin{enumerate}[$(i)$]
\item For every $x \in \XC$, $\tilde{K}(x,\dots, x) = 1$.
\item For every 
$(x^\prime_{1}, \dots x^\prime_{m})
\in \XC^m$, 
$\abs{\tilde{K}(x^\prime_{1}, \dots x^\prime_{m})} \leq 1$.
\end{enumerate}

\end{proposition}
\begin{proof}
Just note that, for every $x \in \XC$, $\norm{\Phi(x)}_m^m = K(x, \cdots, x)>0$.
Then define $\tilde{\phi}_k(x) = \phi_k(x)/\norm{\Phi(x)}_m^m$.
\end{proof}

\begin{remark}\label{rmk:20150527b}
The homogeneous polynomial form \eqref{eq:20150527a} can be written as follows
\begin{equation}\label{eq:20160208a}
\sum_{\substack{\alpha \in \N^n \\ \abs{\alpha}=m}} 
\binom{m}{\alpha} 
K(\underbrace{x_1, \dots, x_1}_{\alpha_1}, \dots, \dots, \underbrace{x_n, \dots , x_n}_{\alpha_n})
 u^\alpha
 \end{equation}
where, for every multi-index $\alpha = (\alpha_1, \dots, \alpha_n) \in \N^n$ and
for every vector $u \in \R^n$, we used the standard notation $u^\alpha = u_1^{\alpha_1} \cdots u_n^{\alpha_n}$,
$\abs{\alpha} = \sum_{i=1}^n \alpha_i$, and the multinomial coefficient
\begin{equation}
\binom{m}{\alpha} = \binom{m}{\alpha_1, \dots, \alpha_n} =  \frac{m!}{\alpha_1!\dots \alpha_n!}.
\end{equation}
Indeed it follows from \eqref{eq:20160209a} and the multinomial theorem 
\cite[Theorem~4.12]{Bona11}
that
\begin{equation*}
\begin{aligned}
 \bigg\lVert \sum_{i=1}^n u_i \Phi(x_i) \bigg\rVert^{r^*}_{r^*}
&=  \sum_{k \in \KK} 
\bigg(\sum_{i=1}^n  u_i \phi_k(x_i) \bigg)^m\\
& = \sum_{k \in \KK} 
\sum_{\substack{\alpha \in \N^n \\ \abs{\alpha}=m}} 
\binom{m}{\alpha}  \phi_k(x_1)^{\alpha_1}\dots \phi_k(x_n)^{\alpha_n} 
u^\alpha \\
& = \sum_{\substack{\alpha \in \N^n \\ \abs{\alpha}=m}} 
\binom{m}{\alpha} 
\bigg(\sum_{k \in \KK} \phi_k(x_1)^{\alpha_1}\dots \phi_k(x_n)^{\alpha_n} 
\bigg) u^\alpha.
\end{aligned}
\end{equation*}
Thus \eqref{eq:20160208a} follows from \eqref{eq:20150525b}.
\end{remark}

We present the  main result of the section,
which is a direct consequence of Theorem~\ref{thm:pridu} and 
Proposition~\ref{p:tensor}.

\begin{theorem}
\label{thm2}
Under the setting \eqref{eq:20160225a}-\eqref{eq:20170309a} described above, 
with $r^*=m$ even integer, 
let $K$ be defined as in \eqref{eq:20150525b}.
%set
%\begin{equation}
%(\forall\, u \in \R^n)\qquad K[x_1, \dots, x_n; u]
%:= \sum_{i_1,\dots, i_m = 1}^n K(x_{i_1}, \dots, x_{i_m}) u_{i_1} \dots u_{i_m}.
%\end{equation}
Then
the dual problem $(\mathcal{D}_n)$ can be written as the following 
finite dimensional optimization problem
\begin{equation}
\label{eq:20150526d}
\left[
\begin{aligned}
&\min_{ u \in \R^n} 
\varphi^*\bigg( \frac{1}{n}
\bigg(  \sum_{i_1,\dots, i_m = 1 }^n
 K(x_{i_1}, \dots, x_{i_m}) 
 u_{i_1} \dots u_{i_m}\bigg)^{1/r^*} \bigg)
+ \frac{\gamma}{n} \sum_{i=1}^n 
L^*\bigg(\frac{u_i}{\gamma}\bigg) - \frac{1}{n}\sum_{i=1}^n y_i u_i\\[1.5ex]
&\text{subject to } \sum_{i=1}^n u_i = 0.
\end{aligned}
\right.
\end{equation}
Moreover, in the related optimality conditions \eqref{eq:optPDn},
the first equation turns to
\begin{equation}
\label{eq:20170313a}
w \in \frac{\partial\varphi^*(\frac{1}{n}\KKK[u]^{1/r^*})}{\KKK[u]^{1/r}}
J_{m} \bigg( \sum_{i=1}^n u_i\Phi(x_i) \bigg), 
%\quad K[u]
%:= \sum_{i_1,\dots, i_m = 1}^n K(x_{i_1}, \dots, x_{i_m}) u_{i_1} \dots u_{i_m},
\end{equation}
%The related optimality conditions \eqref{eq:optPDn} 
%become
%%The dual problem \eqref{eq:20150526d} admits solutions and if $u$ is any of such solutions, then
%%the solutions of the primal problem $(\mathcal{P}_n)$ are given by the following equations
%\begin{equation}
%\label{eq:20170313a}
%\left\{
%\begin{aligned}
%&w \in \frac{\partial\varphi^*(\frac{1}{n}K[u]^{1/r^*})}{K[u]^{1/r}}
%J_{m} \bigg( \sum_{i=1}^n u_i\Phi(x_i) \bigg), \quad K[u]
%:= \sum_{i_1,\dots, i_m = 1}^n K(x_{i_1}, \dots, x_{i_m}) u_{i_1} \dots u_{i_m} \\
%&\sum_{i=1}^n u_i = 0\\
%&u_i/\gamma \in \partial \loss(e_i)
%\quad \text{for every}\ i \in \{1,\dots,n\}\\[1.5ex]
%&y_i - \pair{w}{\Phi(x_i)}_{r,r^*} - b = e_i
%\quad \text{for every}\ i \in \{1,\dots, n\},
%%&b  \in \bigcap_{j=1}^n \bigg[  y_j - \pair{w}{\Phi(x_j)}_{r,r^*} -  
%%\partial L^*\bigg( \frac{u_j}{\gamma} \bigg) \bigg],
%\end{aligned}
%\right.
%\end{equation}
where $\KKK[u]:= \sum_{i_1,\dots, i_m = 1}^n K(x_{i_1}, \dots, x_{i_m}) u_{i_1} \dots u_{i_m}$,
 $J_m\colon \ell^{m}\!(\KK) \to \ell^{r}(\KK)\colon u \mapsto {( u_k^{m-1})}_{k \in \N}$,
and the right hand side  \eqref{eq:20170313a} 
is meant to be $\{0\}$ when $\KKK[u]=0$.
\end{theorem}

\begin{remark}\
\label{rmk:20170315e}
\begin{enumerate}[$(i)$]
\item Problem \eqref{eq:20150526d} is a 
 convex optimization problem.
If the tensor kernel $K$
is explicitly computable by means of \eqref{eq:20150525b},
the dual problem $\eqref{eq:20150526d}$ is a very finite dimensional problem,
in the sense that it does not involve the feature map anymore. This is exactly how
the kernel trick works within the matrix kernels.
\item Once a solution $u \in \R^n$ of the dual problem 
\eqref{eq:20150526d} is computed, the solutions of the primal problem
$(\mathcal{P}_n)$ are given by
\begin{equation}
\label{eq:20170315a}
\left\{
\begin{aligned}
%&w \in \frac{\partial\varphi^*(\frac{1}{n}K[u]^{1/r^*})}{K[u]^{1/r}}
%J_{m} \bigg( \sum_{i=1}^n u_i\Phi(x_i) \bigg),\\
&(\forall\, k \in \N)\quad w_k= \xi(u) 
\bigg( \sum_{i=1}^n u_i \phi_k(x_i) \bigg)^{m-1}\!\!, 
\qquad \xi(u) \in \frac{\partial \varphi^*(\frac{1}{n}\KKK[u]^{1/r^*})}{\KKK[u]^{1/r}},\\
&b  \in \bigcap_{j=1}^n \bigg[  y_j - \pair{w}{\Phi(x_j)}_{r,r^*} -  
\partial L^*\bigg( \frac{u_j}{\gamma} \bigg) \bigg],
\end{aligned}
\right.
\end{equation}
where $\xi(u)=0$ if $\KKK[u]=0$.
%Moreover, when $\varphi$ is strictly convex, ($\varphi^*$ is differentiable on $\R$)
%the first of  \eqref{eq:20170315a} yields
%\begin{equation}
%\label{eq:20160209b}
%(\forall\, k \in \N)\qquad w_k= \xi(u) 
%\bigg( \sum_{i=1}^n u_i \phi_k(x_i) \bigg)^{m-1}, 
%\qquad \xi(u) \in \frac{\partial \varphi^*(\frac{1}{n}K[u]^{1/r^*})}{K[u]^{1/r}}.
%\end{equation}
%\item once a solution $u$ of the dual problem \eqref{eq:20150526d} is found, conditions 
%\eqref{eq:20170313a}
%provide all the primal solutions.
\item\label{rmk:20170315eiii} If there exists $j$
such that $L^*$ is differentiable at $u_j/\gamma$, the second of \eqref{eq:20170315a}
uniquely determines $b$ as 
$b = y_j - \pair{w}{\Phi(x_j)}_{r,r^*} -  
 (L^*)^\prime\big( u_j/\gamma \big)$.
\end{enumerate}
\end{remark}

\begin{corollary}
\label{cor:dualLr}
In Theorem~\ref{thm2}, let $\varphi = (1/r) \abs{\cdot}^r$ 
(which gives $G = (1/r) \norm{\cdot}^r_r$). Then the dual  
problem \eqref{eq:20150526d} becomes
\begin{equation}
\label{eq:20160127b}
\left[
\begin{aligned}
&\min_{ u \in \R^n} 
\frac{1}{r^* n^{r^*}} 
 \sum_{i_1,\dots, i_m = 1 }^n
 K(x_{i_1}, \dots, x_{i_m}) 
 u_{i_1} \dots u_{i_m} 
 + \frac{\gamma}{n} \sum_{i=1}^n L^*\bigg(\frac{u_i}{\gamma}\bigg) 
 - \frac{1}{n} \sum_{i=1}^n y_i u_i\\[1.5ex]
&\text{subject to } \sum_{i=1}^n u_i = 0.
\end{aligned}
\right.
\end{equation}
Moreover, if $u$ is a solution of \eqref{eq:20160127b} and there exists
$j$ such that $L^*$
 is differentiable at $u_j/\gamma$,
then the primal problem $(\mathcal{P}_n)$ has a unique solution which is given by
\begin{equation}
(\forall\, k \in \N)\quad w_k = \frac{1}{n^{m-1}} \bigg( \sum_{i=1}^n u_i \phi_k(x_i) \bigg)^{m-1}
\quad \text{and}\quad b = y_j - \pair{w}{\Phi(x_j)}_{r,r^*} -  
 (L^*)^\prime\bigg( \frac{u_j}{\gamma} \bigg).
\end{equation}
\end{corollary}
\begin{proof}
Just note that $\varphi^* = (1/r^*) \abs{\cdot}^{r^*}$ and apply Theorem~\ref{thm2}
and Remark~\ref{rmk:20170315e}\ref{rmk:20170315eiii}.
\end{proof}

\begin{remark}\
\begin{enumerate}[$(i)$]
\item
The first term in the objective function in \eqref{eq:20160127b} is a positive definite homogeneous polynomial 
of order $m$. So, if the function $L^*$ is smooth,
 which occurs when $L$ is strictly convex,
 then the dual problem \eqref{eq:20160127b} is a smooth convex optimization
 problem with a linear constraint and can be approached 
 by standard optimization techniques such as Newton-type
 or gradient-type
 methods --- in the case of square loss, the dual problem \eqref{eq:20160127b} is 
 a  convex polynomial optimization problem and possibly
 more appropriate optimization methods may be employed.
 \item 
When \eqref{eq:20160127b} is specialized to the case of 
 $\varepsilon$-insensitive loss (see Example~\ref{ex:epsinsensitive})
 we obtain
 \begin{equation}
\label{eq:20160303a}
\left[
\begin{aligned}
&\min_{ u \in \R^n} 
\frac{1}{m n^{m}}  
 \sum_{i_1,\dots, i_m = 1 }^n
 K(x_{i_1}, \dots, x_{i_m}) 
 u_{i_1} \dots u_{i_m} 
 + \frac{\varepsilon}{n} \sum_{i=1}^n \abs{u_i}
 - \frac{1}{n} \sum_{i=1}^n y_i u_i\\[1.5ex]
&\text{subject to } \sum_{i=1}^n u_i = 0\quad\text{and}\quad 
\abs{u_i} \leq \gamma\ \text{for every}\ i \in \{1,\dots,n\},
\end{aligned}
\right.
\end{equation}
which is clearly a generalization of \eqref{eq:20170221b}.
\end{enumerate}
%This problem clearly shows similarities with the dual formulation of  
%standard support vector regression 
%\cite{SteChi2009,Vap98}.
\end{remark}

The next issue is to evaluate the regression function 
 corresponding to $(w,b)$ at a general input point, without
the explicit knowledge of the feature map but relying on the 
tensor-kernel $K$ only. 
%In the analogue case of matrix-kernels, 
%this is what is usually called \emph{kernel trick}.
The following proposition shows that
a tensor-kernel representation holds and hence
 the kernel trick
is fully viable in our more general situation.

\begin{proposition}
\label{p:20160302a}
Under the assumptions \eqref{eq:20160225a}-\eqref{eq:20170309a},
let $K$ be defined as in \eqref{eq:20150525b}.
Suppose that the primal problem $(\mathcal{P}_n)$ has solutions.
%Suppose that $\varphi^*$ is differentiable on $\RPP$.
Let $u \in \R^n$ be a solution of the dual problem \eqref{eq:20150526d} and 
let $(w,b)$ be a solution of $(\mathcal{P}_n)$ determined as in \eqref{eq:20170315a}.
Then, 
\begin{equation}
\label{eq:20160301h}
\left\{
\begin{aligned}
&\pair{w}{\Phi(x)}_{r,r^*} 
= \xi(u)
\sum_{ i_1,\dots, i_{m-1} = 1}^n
K(x_{i_1}, \dots, x_{i_{m-1}}, x)
u_{i_1} \cdots u_{i_{m-1}},\quad \forall\,x \in \XC,\\
&b \in \bigcap_{j=1}^n \bigg[ y_j - \xi(u)
\sum_{ i_1,\dots, i_{m-1} = 1}^n
K(x_{i_1}, \dots, x_{i_{m-1}}, x_j)
u_{i_1} \cdots u_{i_{m-1}}
%&\hspace{10ex}
- \partial L^*\bigg(\frac{u_j}{\gamma}\bigg) \bigg],
\end{aligned}
\right.
\end{equation}
where
 $\xi(u) \in \KKK[u]^{-1/r} \partial \varphi^*(\frac{1}{n}\KKK[u]^{1/r^*})$ if $\KKK[u]\neq 0$
and $\xi(u)=0$ if $\KKK[u]=0$, and $\KKK[u]$ is defined as in Theorem~\ref{thm2}.
\end{proposition}
\begin{proof}
Let $x \in \XC$. Then, we derive from \eqref{eq:20170315a} that
\begin{align*}
 \pair{w}{\Phi(x)}_{r,r^*} 
&=  \sum_{k \in \KK} w_k \phi_k(x) \\
 &= \xi(u) \sum_{k \in \KK} 
\bigg( \sum_{i=1}^n u_i \phi_k(x_i) \bigg)^{m-1}\!\!\! \phi_k(x)\\
& = \xi(u) \sum_{k \in \KK}
\sum_{i_1,\dots, i_{m-1} = 1}^n
\phi_k(x_{i_1}) \cdots \phi_k(x_{i_{m-1}}) \phi_k(x) 
u_{i_1} \cdots u_{i_{m-1}}\\
&= \xi(u) \sum_{i_1,\dots, i_{m-1} = 1}^n
K(x_{i_1}, \dots, x_{i_{m-1}}, x)
u_{i_1} \cdots u_{i_{m-1}},
\end{align*}
where we used the definition \eqref{eq:20150525b} of $K$.
The second equation in \eqref{eq:20160301h} follows from the 
second equation in \eqref{eq:20170315a}.
\end{proof}

\begin{remark}
In the case treated in Corollary~\ref{cor:dualLr}, assuming that
there exists $j$ such that $L^*$ is differentiable at $u_j/\gamma$, 
\eqref{eq:20160301h} yields the following representation formula
\begin{multline*}
\pair{w}{\Phi(x)}_{r,r^*} +b
= \frac{1}{n^{m-1}} \sum_{ i_1,\dots, i_{m-1} = 1}^n
\big( K(x_{i_1}, \dots, x_{i_{m-1}}, x)- K(x_{i_1}, \dots, x_{i_{m-1}}, x_j) \big)
u_{i_1} \cdots u_{i_{m-1}}\\
+ y_j - (L^*)^\prime\bigg(\frac{u_j}{\gamma}\bigg),
\end{multline*}
which generalizes \eqref{eq:20170315c}.
Moreover, if in model \eqref{eq:20160225b} we assume no offset  ($b=0$), then we can avoid the requirement of the differentiability of $L^*$ and the representation formula becomes
\begin{equation*}
\pair{w}{\Phi(x)}_{r,r^*} 
= \frac{1}{n^{m-1}} \sum_{ i_1,\dots, i_{m-1} = 1}^n
 K(x_{i_1}, \dots, x_{i_{m-1}}, x)
u_{i_1} \cdots u_{i_{m-1}}.
\end{equation*}
\end{remark}

Concluding we have shown that, the estimated regression function
can be evaluated at every point of the input space by means of 
a finite summation formula, provided that
 the tensor-kernel $K$ is explicitly available: we will show in Section~\ref{secPSTK} several
significant examples in which this occurs.

\subsection{The complex case}
\label{sec:complex}
In this section we give the complex version of the theory developed in 
Section~\ref{subsec:real}.
Therefore, we let $\displaystyle\FF = \ell^r(\KK;\C)$, with $\KK$ a countable set and 
$r = m/(m-1)$ for some even integer  $m\geq 2$. 
Let $(\phi_k)_{k \in \KK}$ be a family of measurable functions from $\XC$ to $\C$ such that,
for every $x \in \XC$, $\displaystyle {(\phi_k(x))}_{k \in \KK} \in \ell^{r^*}\!(\KK;\C)$.
The feature map is now defined as
\begin{equation}
\label{eq:20160226a}
\Phi\colon \XC \to \ell^{r^*}\!(\KK;\C)\colon x \mapsto {(\overline{\phi_k(x)})}_{k \in \KK},
\end{equation}
which generates the model 
\begin{equation}
(\forall\, w \in \ell^r(\KK;\C))(\forall\, b \in \C)\qquad 
x\mapsto \pair{w}{\Phi(x)}_{r,r^*} + b= \sum_{k \in \KK} w_k \phi_k(x) + b,
\end{equation}
where $\pair{w}{w^*}_{r,r^*} = \sum_{k \in \N} w_k \overline{w_k^*}$ is the
canonical sesquilinear form between $\ell^r(\KK;\C)$ and $\ell^{r^*}(\KK;\C)$.
This case can be treated as a vector-valued real case
by identifying complex functions with $\R^2$-valued functions and 
the space
$\ell^r(\KK;\C)$ with $\ell^r(\KK;\R^2)$.
Moreover, it is not difficult to generalize the dual framework
presented in Section~\ref{sec:GenSVR} to the case of vector-valued  
(and specifically to $\R^2$-valued) functions.
Then, the (complex) feature map \eqref{eq:20160226a} defines an underlying real 
vector-valued feature map on $\ell^r(\KK;\R^2)$ \cite{Com15}, that is
\begin{equation}
\Phi_\R\colon \XC \to \mathcal{L}(\R^2,\ell^{r^*}\!(\KK;\R^2)) 
\approxeq \ell^{r^*}\!(\KK;\R^{2\times 2})\colon x \mapsto {(\phi_{\R,k}(x))}_{k \in \KK},
\end{equation}
where $\mathcal{L}(\R^2,\ell^{r^*}\!(\KK;\R^2))$ is the spaces of linear continuous operators
from $\R^2$ to $\ell^{r^*}\!(\KK;\R^2)$ (which is isomorphic to $\ell^{r^*}\!(\KK;\R^{2\times 2})$)
and
\begin{equation}
(\forall\, x \in \XC)(\forall\, k \in \KK)\quad \phi_{\R,k}(x) =
\begin{bmatrix}
\mathfrak{Re} \phi_k(x) & \mathfrak{Im} \phi_k(x)\\
- \mathfrak{Im} \phi_k(x) & \mathfrak{Re} \phi_k(x)
\end{bmatrix} \in \R^{2\times 2}.
\end{equation}
This way, denoting, for every $x \in \XC$, by $\phi_{\R,k}(x)^*$ 
the transpose of the matrix $\phi_{\R,k}(x)$, we have
\begin{equation}
(\forall\, x \in \XC)(\forall\, k \in \KK)(\forall\, w_k \in \R^2\approxeq \C)\qquad \phi_{\R,k}(x)^* w_k 
= w_k \phi_k(x),
\end{equation}
hence $\Phi_\R(x)^* w = \pair{w}{\Phi(x)}_{r,r^*}$.
Moreover
\begin{equation}
(\forall\, x \in \XC)(\forall\, u \in \R^2\approxeq \C)\qquad
\Phi_\R(x) u = {(\phi_{\R,k}(x) u)}_{k \in \KK} 
 = {(u \overline{\phi_k(x)})}_{k \in \KK}  = u \Phi(x).
\end{equation}
Then, problems $(\mathcal{P}_n)$ and $(\mathcal{D}_n)$ become
\begin{equation}
\tag{$\mathcal{P}_n(\C)$}
\left[
\begin{aligned}
&\min_{(w,b,e) \in \ell^r(\KK;\C)\times\C \times \C^n }
\frac{\gamma}{n} \sum_{i=1}^n  \loss(e_i) + \varphi(\norm{w}_r),\\[1ex]
&\text{subject to}\ \ 
y_i -\pair{w}{\Phi(x_i)}_{r,r^*} - b  = e_i,
\quad \text{for every}\ i \in \{1,\dots, n\}
\end{aligned}
\right.
\end{equation}
and
\begin{equation}
\tag{$\mathcal{D}_n(\C)$}
\left[
\begin{aligned}
&\min_{ u \in \C^n} 
\varphi^*\bigg(\bigg\lVert \frac 1 n \sum_{i=1}^n u_i \Phi(x_i) \bigg\rVert_{r^*} \bigg) 
+ \frac{\gamma}{n} \sum_{i=1}^n L^*\bigg(\frac{u_i}{\gamma}\bigg) 
- \frac{1}{n} \sum_{i=1}^n \mathfrak{Re} (u_i \overline{y_i})\\[1.5ex]
&\text{subject to } \sum_{i=1}^n u_i = 0,
\end{aligned}
\right.
\end{equation}
where,
$L^*\colon \C \to \R\colon z^* \mapsto \sup_{z \in \C} \mathfrak{Re}(z \overline{z^*}) - L(z)$.
Moreover,
assuming that $w\neq 0$, the optimality conditions \eqref{eq:optPDn} still hold,
where now 
$J_{r^*} \colon \ell^{r^*}(\KK;\C) \to \ell^{r}\!(\KK;\C)\colon 
w^* \mapsto {(\abs{w^*_k}^{r-1} w^*_k/\abs{w^*_k})}_{k \in \KK}$,
and 
\begin{equation*}
(\forall\, e \in \C)\quad \partial L(e) = \big\{z^*\in \C~\big\vert~(\forall\, z \in \C)\ L(z) \geq L(e) 
+ \mathfrak{Re} \big(\overline{z^*}(z - e) \big)  \big\}.
\end{equation*}

In the following we give the result corresponding to Proposition~\ref{p:tensor}.

\begin{proposition}
\label{p:tensorC}
In the setting described above, suppose that $m$ is even and set $q=m/2$.
Then, the following function is well-defined
\begin{equation}
\label{eq:20150525bC}
K\colon \XC^q \times \XC^q \to \C\colon 
(x^\prime_1, \dots, x^\prime_q; x^{\prime\prime}_1, \dots, x^{\prime\prime}_q) 
\mapsto \sum_{k \in \KK}  
\phi_k(x^\prime_{1})\cdots \phi_k(x^\prime_{q})
\overline{\phi_k(x^{\prime\prime}_{1})}\cdots \overline{\phi_k(x^{\prime\prime}_{q})},
\end{equation}
and the following hold.
\begin{enumerate}[$(i)$]
\item\label{p:tensorCii}
For every 
$(x^\prime_1, \dots, x^\prime_q; x^{\prime\prime}_1, \dots, x^{\prime\prime}_q) 
\in \XC^q \times \XC^q$, and for every permutation $\sigma^\prime$ and 
$\sigma^{\prime\prime}$ of the indexes $\{1,\dots, q\}$,
\begin{equation*}
K(x^\prime_{\sigma^\prime(1)} \dots x^\prime_{\sigma^\prime(q)}; 
x^{\prime\prime}_{\sigma^{\prime\prime}(1)} \dots x^{\prime\prime}_{\sigma^{\prime\prime}(q)}) = 
K(x^\prime_{1}, \dots x^\prime_{q} ; x^{\prime\prime}_{1} \dots x^{\prime\prime}_{q}).
\end{equation*}
\item\label{p:tensorCiii}
For every 
$(x^\prime; x^{\prime\prime}) \in \XC^q\times \XC^q$
$K(x^\prime; x^{\prime\prime}) = \overline{K(x^{\prime\prime}; x^\prime)}$;
\item\label{p:tensorCvi}
For every $(x_i)_{1 \leq i \leq n} \in \XC^n$
\begin{equation*}
(\forall\, u \in \C^n)
\quad\sum_{\substack{i_1,\dots, i_q = 1 \\ \ii_1,\dots, \ii_q = 1}}^n
 K(x_{\ii_1}, \dots, x_{\ii_q};x_{i_1}, \dots, x_{i_q}) 
 u_{i_1} \dots u_{i_q} \overline{u_{\ii_1}} \dots \overline{u_{\ii_q}} \geq 0\,.
\end{equation*}
\item\label{p:tensorCv}
For every $(x_i)_{\ \leq i \leq n} \in \XC^n$
\begin{equation*}
u \in \C^n \mapsto \bigg\lVert \sum_{i=1}^n u_i \Phi(x_i) \bigg\rVert^{r^*}_{r^*}
= \sum_{\substack{i_1,\dots, i_q = 1 \\ \ii_1,\dots, \ii_q = 1}}^n
 K(x_{\ii_1}, \dots, x_{\ii_q}; x_{i_1}, \dots, x_{i_q}) 
 u_{i_1} \dots u_{i_q} \overline{u_{\ii_1}} \dots \overline{u_{\ii_q}}
\end{equation*}
is a positive homogeneous polynomial form of degree $m$ on $\C^n$.
\item\label{p:tensorCi} 
For every $(x^\prime_1, \dots, x^\prime_q) \in \XC^q$, 
$K(x^\prime_1, \dots, x^\prime_q;x^\prime_1, \dots, x^\prime_q) \geq 0$;
\item\label{p:tensorCiv} 
For every $(x^\prime_1, \dots, x^\prime_q; x^{\prime\prime}_1, \dots, x^{\prime\prime}_q)  
\in \XC^q \times \XC^q$,
\begin{equation*}
\abs{K(x^\prime_1, \dots, x^\prime_q; x^{\prime\prime}_1, \dots, x^{\prime\prime}_q)} 
\leq K(x^\prime_1, \dots, x^\prime_1; x^\prime_1, \dots, x^\prime_1)^{1/m} 
\cdots K(x^{\prime\prime}_q, \dots, x^{\prime\prime}_q; x^{\prime\prime}_q, \dots, x^{\prime\prime}_q)^{1/m}.
\end{equation*}
\end{enumerate}
\end{proposition}

\begin{remark}
Item \ref{p:tensoriv} states that 
$\KKK:=\big(K(x_{i_1}, \dots, x_{i_m})\big)_{i \in \{1, \dots, n\}^m }$ is a
positive-definite tensor of degree $m$.
\end{remark}

As in the real case, the dual problem $(\mathcal{D}_n)$ reduces to
\begin{equation*}
\left[
\begin{aligned}
&\min_{ u \in \C^n} 
\varphi^*\bigg( \frac{1}{n}
\bigg(  \sum_{\substack{i_1,\dots, i_q = 1 \\ \ii_1,\dots, \ii_q = 1}}^n
 K(x_{\ii_1}, \dots, x_{\ii_q}, x_{i_1}, \dots, x_{i_q}) 
 u_{i_1} \dots u_{i_q} \overline{u_{\ii_1}} \dots \overline{u_{\ii_q}}\bigg)^{1/r^*} \bigg)\\
&\hspace{50ex}+ \frac{\gamma}{n} \sum_{i=1}^n L^*\bigg(\frac{u_i}{\gamma}\bigg) 
- \frac{1}{n}\mathfrak{Re}\sum_{i=1}^n \overline{y_i} u_i\\[1.5ex]
&\text{subject to } \sum_{i=1}^n u_i = 0
\end{aligned}
\right.
\end{equation*}
and 
the homogeneous polynomial form in Proposition~\ref{p:tensorC}\ref{p:tensorCv} can be written as follows
\begin{equation}
\label{eq:20160301e}
\sum_{\substack{\alpha \in \N^n, \beta \in \N^n \\ \abs{\alpha}=q, \abs{\beta}=q}} 
\binom{q}{\alpha} \binom{q}{\beta}  
K(\underbrace{x_1, \dots x_1}_{\alpha_1}, \dots \dots, \underbrace{x_n, \dots , x_n}_{\alpha_n},
\underbrace{x_1, \dots x_1}_{\beta_1}, \dots \dots, \underbrace{x_n, \dots , x_n}_{\beta_n})
 u^\alpha\, \overline{u}^\beta.
 \end{equation}
Finally, in the setting of Proposition~\ref{p:20160302a}, defining
\begin{equation}
\KKK[u] = \sum_{\substack{i_1,\dots, i_q = 1 \\ \ii_1,\dots, \ii_q = 1}}^n
 K(x_{\ii_1}, \dots, x_{\ii_q}, x_{i_1}, \dots, x_{i_q}) 
 u_{i_1} \dots u_{i_q} \overline{u_{\ii_1}} \dots \overline{u_{\ii_q}},
\end{equation}
for every $x \in \XC$, the following representation formulas hold
\begin{equation}
\label{eq:20160208d}
\begin{aligned}
\pair{w}{\Phi(x)}_{r,r^*} 
&= \xi(u)
\sum_{\substack{ i_1,\dots, i_{q} = 1\\\ii_1,\dots, \ii_{q-1} = 1}}^n
K(x_{\ii_1}, \dots x_{\ii_{q-1}},x; x_{i_1}, \dots x_{i_{q}})
u_{i_1} \cdots u_{i_{q}}
\overline{u_{\ii_1}} \cdots \overline{u_{\ii_{q-1}}}\\
\xi(u)&\in \frac{\partial \varphi^*( \frac{1}{n}\KKK[u]^{1/r^*})}{\KKK[u]^{1/r}}\\
b&= y_j - \pair{w}{\Phi(x_j)}_{r,r^*}  - \nabla L^*\bigg(\frac{u_j}{\gamma}\bigg).
\end{aligned}
\end{equation}
where we assumed that $L^*$ is differentiable, as a function from $\R^2$ to $\R$, at some $x_j/\gamma$.

\begin{remark}
In view of Proposition~\ref{p:tensorC}\ref{p:tensorCv}, definitions
\eqref{eq:20160226a} and \eqref{eq:20150525bC} 
correspond to those given in \cite[Lemma~4.2]{SteChi2008}
and the concept of positive definiteness stated in \ref{p:tensorCvi} is a natural 
generalization of
the analogue notion given in \cite[Definition~4.15]{SteChi2008}.
\end{remark}

\section{Power series tensor-kernels}
\label{secPSTK}

In this section we consider reproducing kernel Banach spaces
of complex  analytic functions
which are generated through power series. 
We show that, for such spaces,
 the corresponding tensor kernel, defined according to \eqref{eq:20150525b},
admits an explicit expression. We provide also representation formulas.
In this section we assume, for simplicity, that $\varphi = (1/r)\abs{\cdot}^r$,
with $r=m/(m-1)$ for some even integer $m \geq 2$.
therefore we address the support vector regression problem
\begin{equation*}
\min_{(w,b) \in \ell^r(\KK;\C)\times\C }
\frac{\gamma}{n} \sum_{i=1}^n  \loss\big(y_i -\pair{w}{\Phi(x_i)}_{r,r^*} - b \big) 
+ \frac 1 r \norm{w}_r^r,
\end{equation*}
for a specific choice of the feature map \eqref{eq:20160226a}.

We first need to set special notation for multi-index powers of complex vectors.
Let $d\in \N$ with $d\geq 1$. 
We will denote the component of a vector $x \in \C^d$,
by $x_t$, with $t \in \{1,\dots, d\}$.
For every $x \in \C^d$ and every $\nu \in \N^d$ we set 
\begin{equation*}
x^\nu = \prod_{t=1}^d x_t^{\nu_t}, \quad 
\abs{x} =(\abs{x_1}, \dots, \abs{x_d}),\quad
\text{and}\quad \nu! = \prod_{t=1}^d \nu_t!
\end{equation*}
so that $\forall\, \nu \in \N^d$ we have 
$\abs{x^\nu} = \prod_{t=1}^d \abs{x_t}^{\nu_t} = \abs{x}^\nu$.
Moreover, when the exponent of the vector $x \in \C^d$
is an index (not a multi-index), say $m \in \N$, we consider $m$ as a constant multi-index,
that is $(m,\dots, m)$, so that  $x^m$ means $\prod_{t=1}^d x_t^m$.
Finally, we define the binary inner operation of component-wise multiplication in $\C^d$. 
For every $x, x^\prime \in \C^d$, we define $x \odot x^\prime\in \C^d$ such that, 
for every $t \in \{1,\dots, d\}$, 
$(x \odot x^\prime)_t= x_t x^\prime_t$.
Let $m \in \N$ and $x \in \C^d$. We set $x^{\odot m} = x \odot\cdots \odot x$ ($m$-times),
so  that $x^{\odot m} \in \C^d$ and, for every $t \in \{1, \dots, d\}$, 
$(x^{\odot m})_t = x_t^m$.

Let $\rho = (\rho_\nu)_{\nu \in \N^d}$ be a multi-sequence in $\R_+$
and let
 $\XD_{\rho}$ be the domain of (absolute) convergence of the power series
$\sum_{\nu \in \N^d}  \rho_\nu z^\nu$, i.e., the interior of the set
$\big\{ z \in \C^d~\big\vert~\sum_{\nu \in \N^d} 
\rho_\nu \abs{z^\nu}< + \infty\big\}$.
The set $\XD_{\rho}$ is a complete Reinhardt domain\footnote{
It means that if $z \in \XD_\rho$, then $\XD_\rho$ contains 
the polydisk $\{t \in \C^d~\vert~(\forall\, j \in \{1,\dots, d\})\ \abs{t_j} \leq \abs{z_j}\}$.
}
and we assume that $\XD_{\rho} \neq \{0\}$.
Let $\kappa\colon \XD_{\rho} \to \C$ be the sum of the series 
$\sum_{\nu \in \N^d}  \rho_\nu z^\nu$, that is
\begin{equation*}
(\forall\, z \in \XD_{\rho})\quad \kappa(z) = \sum_{\nu \in \N^d}  \rho_\nu z^\nu.
\end{equation*}
Clearly $\kappa$
is an analytic function on $\XD_{\rho}\subset \C^d$. Set
\begin{equation*}
\XD_{\rho}^{\odot 1/m} 
= \big\{ x\in \C^d~\big\vert~x^{\odot m}= (x_1^m, \dots, x_d^m) \in\XD_{\rho} \big\},
\end{equation*}
let 
 $\XC\subset \XD_{\rho}^{\odot 1/m}$, and define
the dictionary
\begin{equation}
\label{eq:20160303b}
(\forall\, \nu \in \N^d)\qquad\phi_\nu\colon \XC \to \C\colon x \mapsto \rho_\nu^{1/m} x^\nu.
\end{equation}
Then, for every $x \in \XC$, since $x^{\odot m} \in \XD_{\rho}$, we have
\begin{equation*}
\sum_{\nu \in \N^d} \abs{\phi_\nu(x)}^m 
= \sum_{\nu \in \N^d} \rho_\nu {\abs{x^{\odot m}}}^{\nu}<+\infty,
\end{equation*}
hence $(\phi_\nu(x))_{\nu \in \N^d} \in \ell^{m}(\N^d;\C)$.
Thus, we are in the framework described at the beginning of 
Section~\ref{sec:complex}.
We define
\begin{equation*}
B_{\rho, b}^r(\XC) = \bigg\{ f \in \C^\XC \,\bigg\vert\, 
(\exists\, (c_\nu)_{\nu \in {\N}^d} \in \ell^r(\N^d;\C))(\exists\, b \in \C)(\forall\, x \in \XC) 
\Big(f(x) = \sum_{\nu \in \N^d} 
 c_\nu  \phi_\nu(x) + b \Big) \bigg\},
\end{equation*}
which is a reproducing kernel Banach spaces with norm
\begin{equation*}
\norm{f}_{B_{\rho,b}^r(\XC)} = \inf\Big\{  \norm{c}_r + \abs{b}~
\Big\vert~{(c_\nu)}_{\nu \in \N^d} \in 
\ell^r(\N^d;\C)\ \text{and}\ f=\sum_{\nu\in \N^d}  c_\nu \rho_\nu^{1/m} x^{\nu} + b\ \text{(pointwise)} \Big\}.
\end{equation*}
Suppose now that $b=0$ and that, for every $\nu \in \N^d$, $\rho_\nu>0$. Then,
defining the weights $(\eta_\nu)_{\nu \in \N^d} = (\rho_\nu^{-r/m})_{\nu \in \N^d}$
and the corresponding weighted $\ell^r$ space
\begin{equation*}
\ell^r_{\eta}(\N^d;\C) = \bigg\{ {(a_\nu)}_{\nu \in \N^d} \in \C^{\N^d}~\Big
\vert~\sum_{\nu \in \N^d} \frac{1}{\rho_\nu^{r/m}} \abs{a_\nu}^r<+\infty\bigg\},
\end{equation*}
we can express the space $B_{\rho,0}^r(\XC)$ in the form of a weighted
Hardy-like space \cite{Paul,You01}
\begin{equation*}
B_{\rho,0}^r(\XC) = \bigg\{ f \in \C^\XC \,\bigg\vert\, 
(\exists\, (a_\nu)_{\nu \in {\N}^d} \in \ell_\eta^r(\N^d;\C))(\forall\, x \in \XC) 
\Big(f(x) = \sum_{\nu \in \N^d} 
 a_\nu x^\nu \Big) \bigg\}.
\end{equation*}
Moreover, for every $(x^\prime_1,\dots, x^\prime_q, 
x^{\prime\prime}_1,\dots, x^{\prime\prime}_q) \in \XC^q\times \XC^q$,
\begin{equation}
\label{eq:20160301f}
K(x^\prime_1,\dots, x^\prime_q ;
x^{\prime\prime}_1,\dots, x^{\prime\prime}_q) 
= \sum_{\nu \in \N^d} \rho_\nu
x^{\prime \nu}_1\cdots x^{\prime \nu}_q
\overline{x^{\prime\prime \nu}_1}\cdots \overline{x^{\prime\prime \nu}_q}  
= \kappa(x^\prime_1\odot \cdots \odot x^\prime_q \odot 
\overline{x^{\prime\prime}_1}\odot \cdots \odot \overline{x^{\prime\prime}_q}).
\end{equation}

\begin{remark}
Suppose that $\rho_\nu>0$, for every $\nu \in \N^d$. Then
 $\sum_{\nu \in \N^d} c_\nu \rho_\nu^{1/m} x^{\nu} = 0$ (pointwise) implies
 $c_\nu \rho_\nu^{1/m}=0$, for every $\nu \in \N^d$ and hence 
 $c_\nu=0$, for every $\nu \in \N^d$. Thus, in virtue of Remark~\ref{rmk:schauder}
 this yields that 
 $(\phi_\nu)_{\nu \in \N^d}$ is an unconditional Schauder basis of $B_{\rho,0}^r(\XC)$
and that $B_{\rho,0}^r(\XC)$ is isometric to $\displaystyle \ell^r(\N^d;\C)$.
\end{remark}

\begin{proposition}
Under the notation and assumption above, 
suppose that $\XC$ is a compact subset of $D_{\rho}^{\odot 1/m}$
and that, for every $\nu \in \N^d$, $\rho_\nu>0$.
Then $B_{\rho,b}^r(\XC)$ is dense in $\mathscr{C}(\XC;\C)$, the space of continuous functions
on $\XC$ endowed with the uniform norm.
\end{proposition}
\begin{proof}
It is enough to note that $B_{\rho,b}^r(\XC)$ contains the set
\begin{equation*}
\mathcal{A} = \mathrm{span}\big\{ \phi_\nu~\big\vert~\nu \in \N\big\}
= \bigg\{ \sum_{\nu \in I} c_\nu x^\nu~\Big\vert~I\subset \N^d\ \text{and}\ I\ \text{finite}
\ {(c_\nu)}_{\nu \in I}\in \C^{I} \bigg\}
\end{equation*}
which is the algebra of polynomials on $\XC$ in $d$ variables with complex coefficients.
Thus the statement is a consequence of the Stone-Weierstrass theorem.
\end{proof}

In the sequence we also assume that the offset $b$ is zero.
Because of \eqref{eq:20160301f}, the representation given
in \eqref{eq:20160301e} yields the following homogenous polynomial form
\begin{equation}
\label{eq:20160301b}
u \in \C^n \mapsto \bigg\lVert \sum_{i=1}^n u_i \Phi(x_i) \bigg\rVert^{r^*}_{r^*} =
\sum_{\substack{\alpha \in \N^n, \beta \in \N^n\\ \abs{\alpha}=q, \abs{\beta}=q}} 
\binom{q}{\alpha} \binom{q}{\beta} 
\kappa(x_1^{\odot \beta_1}\odot\cdots \odot x_n^{\odot \beta_n} \odot
\overline{x}_1^{\odot \alpha_1}\odot\cdots \odot \overline{x}_n^{\odot \alpha_n}) 
u^\alpha \overline{u}^\beta,
\end{equation}
where $(x_i)_{1 \leq i \leq n} \in \XC^n$ is the training set
and, according to the convention established at the beginning of the section, 
$x_i^{\odot \alpha_i} =( x_{i,1}^{\alpha_i}, \dots, x_{i,d}^{\alpha_i})$.
Moreover, in this case, recalling \eqref{eq:20160208d} 
and \eqref{eq:20160301f}, for every $x \in \XC$,
we have
\begin{equation}
\label{eq:20160301c}
\pair{w}{\Phi(x)}_{r,r^*}
= \frac{1}{n^{m-1}}  \sum_{\substack{i_1,\dots, i_q = 1 \\ \ii_1,\dots, \ii_{q-1} = 1}}^n
\kappa(x_{\ii_1}\odot\cdots \odot x_{\ii_{q-1}} \odot x 
\odot \overline{x_{i_1}}\odot\cdots \odot\overline{x_{i_q}})
u_{i_1} \cdots u_{i_{q}} \overline{u_{\ii_1}} \cdots \overline{u_{\ii_{q-1}}}.
\end{equation}

We now treat two special cases of power series tensor-kernels
built from a power series of one complex variable.
Let $(\gamma_k)_{k \in \N} \in \RP^\N$ and
suppose that the power series 
\begin{equation}
\label{eq:20170313b}
\sum_{k \in \N} \gamma_k \zeta^k\quad (\zeta \in \C)
\end{equation}
has
radius of convergence $R_\gamma>0$ ($R_\gamma = 1/ \limsup_{k} \gamma_k^{1/k}>0$).
We denote by $D(R_\gamma) = \{\zeta \in \C~\vert~ \abs{\zeta}<R_\gamma\}$ and
 by $\psi \colon D(R_\gamma) \to \R$ respectively the disk of convergence
and the sum of the power series \eqref{eq:20170313b}.
\begin{description}
\item[Case 1.]
We set
\begin{equation}
\label{eq:case1C}
(\forall\, \nu \in \N^d)\qquad  \rho_\nu = \gamma_{\abs{\nu}}\binom{\abs{\nu}}{\nu} = \gamma_{\abs{\nu}} \frac{ \abs{\nu}!}{\nu_1!\cdots \nu_d!}.
\end{equation}
Then,
 the domain of absolute convergence of the series 
 $\sum_{\nu \in \N^d} \rho_\nu z^\nu$
 is the strip 
 \begin{equation*}
\XD_{\rho} = \bigg\{ z \in \C^d~\Big\vert~\bigg\lvert \sum_{t=1}^d z_t \bigg\rvert < R_\gamma\bigg\}
\end{equation*}
and,
 it follows from the multinomial theorem \cite[Theorem~4.12]{Bona11} that,
 for every $z \in \XD_\rho$,
 \begin{equation}
\label{eq:20160217c}
\kappa(z) = \sum_{\nu \in \N^d} \rho_\nu z^\nu
=\sum_{k \in \N} \gamma_k 
\sum_{\substack{\nu \in \N^d\\ \abs{\nu} = k}} \frac{k!}{\nu_1! \cdots \nu_d!} 
z^{\nu}
= \sum_{k \in \N} \gamma_k
\bigg( \sum_{t=1}^d z_t \bigg)^k
= \psi\bigg( \sum_{t=1}^d z_t\bigg).
\end{equation}
Note also that $\XD_{\rho}^{\odot 1/m} = \{z \in \C^d~\vert~\norm{z}_m^m<R_\gamma\}$.
Thus, it follows from \eqref{eq:20160301f} that
\begin{align}
\label{eq:20160301j}
K(x^\prime_1,\dots, x^\prime_q;
x^{\prime\prime}_1,\dots, x^{\prime\prime}_q)
\nonumber &= 
\kappa(x^\prime_1\odot \cdots \odot x^\prime_q \odot 
\overline{x^{\prime\prime}_1}\odot \cdots \odot \overline{x^{\prime\prime}_q})\\
&= \psi \bigg(\sum_{t=1}^d x^\prime_{1,t}  \cdots x^\prime_{q,t}
\overline{x^{\prime\prime}_{1,t} }
\cdots \overline{x^{\prime\prime}_{q,t}} \bigg),
\end{align}
for every $(x^\prime_1,\dots, x^\prime_q, 
x^{\prime\prime}_1,\dots, x^{\prime\prime}_q) \in \XC^q\times \XC^q$.
For $q=1$, the right hand side of 
\eqref{eq:20160301j} reduces to
\begin{equation*}
K(x^\prime,x^{\prime\prime}) =
 \psi(\scal{x^\prime}{x^{\prime\prime}}) = \sum_{k \in \N} \gamma_k 
{\scal{x^\prime}{x^{\prime\prime}}}^k,
\end{equation*}
where $\scal{\cdot}{\cdot}$ is the Euclidean scalar product in $\R^d$.
These kind of kernels have been also called \emph{Taylor kernels} in \cite{SteChi2008}.
Thus, in virtue of \eqref{eq:20160301j}, \eqref{eq:20160301b} takes the form
\begin{equation*}
\begin{aligned}
u \in \C^n \mapsto \bigg\lVert \sum_{i=1}^n u_i \Phi(x_i) \bigg\rVert^{r^*}_{r^*} 
&= \sum_{\substack{\alpha \in \N^n, \beta \in \N^n\\ \abs{\alpha}=q, \abs{\beta}=q}} 
\binom{q}{\alpha}\binom{q}{\beta}
\psi\bigg( \sum_{t = 1}^d \overline{x}_{1,t}^{\alpha_1} \cdots \overline{x}_{n,t}^{\alpha_n}
x_{1,t}^{\beta_1} \cdots x_{n,t}^{\beta_n}\bigg) u^\alpha \overline{u}^{\beta}\\
&= \sum_{\substack{\alpha \in \N^n, \beta \in \N^n\\ \abs{\alpha}=q, \abs{\beta}=q}} 
\binom{q}{\alpha}\binom{q}{\beta}
\psi\bigg( \sum_{t = 1}^d (\overline{x}_{\cdot,t})^{\alpha} (x_{\cdot,t})^{\beta}\bigg) 
u^\alpha \overline{u}^\beta,
\end{aligned}
\end{equation*}
where we put, for every $t \in \{1,\dots, d\}$,
$x_{\cdot,t} = ( x_{1,t},  \dots x_{n,t}) \in \C^n$.\footnote{
If we consider the matrix of the data 
$\XX = (x_{i,t})_{\substack{1 \leq i \leq n\\1 \leq t \leq d}} \in \C^{n\times d}$,
having the training set $(x_i)_{1 \leq i \leq n}$ as rows, the vectors $x_{\cdot,t}$
are the columns of $\XX$.} 
The representation formula \eqref{eq:20160301c} turns to
\begin{equation*}
\begin{aligned}
\pair{w}{\Phi(x)}_{r,r^*}  
&= \frac{1}{n^{m-1}}  \sum_{\substack{i_1,\dots, i_{q} = 1\\ \ii_1,\dots, \ii_{q-1} = 1}}^n  
\psi \bigg(\sum_{t=1}^d \overline{x_{i_1,t}} \cdots \overline{x_{i_{q},t}} 
x_{\ii_1,t} \cdots x_{\ii_{q-1},t} x_t \bigg) 
u_{i_1} \cdots u_{i_{q}}
\overline{u_{\ii_1}} \cdots \overline{u_{\ii_{q-1}}}.
\end{aligned}
\end{equation*}

\item[Case 2.] 
We set
\begin{equation}
\label{eq:case2C}
(\forall\, \nu \in \N^d)\qquad \rho_\nu = \prod_{t=1}^d \gamma_{\nu_t}.
\end{equation}
Then
 the domain of absolute convergence of the series 
 $\sum_{\nu \in \N^d} \rho_\nu z^\nu$ is
 \begin{equation*}
\XD_\rho = \bigg\{ z \in \C^d~\Big\vert~(\forall\, t \in \{1,\dots, d\})\ \abs{z_t}<R_\gamma\bigg\}
\end{equation*}
 and
\begin{equation*}
(\forall\, z \in \XD_\rho)\quad \kappa(z) =
\sum_{\nu \in \N^d} \rho_\nu z^\nu = 
\sum_{\nu \in \N^d} \prod_{t=1}^d \gamma_{\nu_j} z_t^{\nu_t} 
= \prod_{t=1}^d \sum_{k \in \N} \gamma_k z_t^k = \prod_{t=1}^d \psi(z_t).
\end{equation*}
In this case $\XD_{\rho}^{\odot 1/m} 
= \{z \in \C^d~\vert~(\forall\, t \in \{1,\dots, d\})\abs{z_t}<R^{1/m}_\gamma\}$ and
 \eqref{eq:20160301f} becomes,
\begin{align}
\label{eq:20160217a}
K(x^\prime_1,\dots, x^\prime_q;
x^{\prime\prime}_1,\dots, x^{\prime\prime}_q)
\nonumber &= 
\kappa(x^\prime_1\odot \cdots \odot x^\prime_q \odot 
\overline{x^{\prime\prime}_1}\odot \cdots \odot \overline{x^{\prime\prime}_q})\\
&= \prod_{t=1}^d \psi \bigg( x^\prime_{1,t}  \cdots x^\prime_{q,t}
\overline{x^{\prime\prime}_{1,t}}
\cdots \overline{x^{\prime\prime}_{q,t}} \bigg),
\end{align}
for every $(x^\prime_1,\dots, x^\prime_q, 
x^{\prime\prime}_1,\dots, x^{\prime\prime}_q) \in \XC^q\times \XC^q$.
Thus, as done before, relying on \eqref{eq:20160217a} we can obtain the
corresponding expression for the homogeneous polynomial form \eqref{eq:20160301b} 
\begin{equation}
\label{eq:20160217e}
u \in \C^n \mapsto \bigg\lVert \sum_{i=1}^n u_i \Phi(x_i) \bigg\rVert^{r^*}_{r^*} 
= \sum_{\substack{\alpha \in \N^n, \beta \in \N^n\\ \abs{\alpha}=q, \abs{\beta}=q}} 
\binom{q}{\alpha}\binom{q}{\beta} \prod_{t=1}^d
\psi\big(  \overline{x}_{1,t}^{\alpha_1} \cdots \overline{x}_{n,t}^{\alpha_n}
x_{1,t}^{\beta_1} \cdots x_{n,t}^{\beta_n}\big) u^\alpha \overline{u}^\beta
\end{equation}
and the representation formula \eqref{eq:20160301c},
\begin{equation}
\label{eq:20160217f}
\pair{w}{\Phi(x)}_{r,r^*} 
= \frac{1}{n^{m-1}}  \sum_{\substack{i_1,\dots, i_{q} = 1 \\ \ii_1,\dots, \ii_{q-1} = 1}}^n
\prod_{t=1}^d \psi( x_{{\ii_1},t}\cdots x_{{\ii_{q-1}},t} x_t 
\overline{x_{{i_1},t}}\cdots \overline{x_{{i_{q},t}}} )
u_{i_1} \cdots u_{i_{q}} \overline{u_{\ii_1}} \cdots \overline{u_{\ii_{q-1}}}.
\end{equation}
\end{description}

\begin{example} We list significant examples of power series 
tensor kernels and for each one we provide the corresponding representation formulas.
\begin{enumerate}[$(i)$]
\item In \eqref{eq:case2C} set $(\gamma_k)_{k \in \N}\equiv 1$,
hence $(\rho_\nu)_{\nu \in \N^d}\equiv 1$ too. Then $R_\gamma = 1$ and
$\psi(\zeta) = 1/(1 - \zeta)$. Therefore, relying on \eqref{eq:20160217a}, we obtain
the tensor-\emph{Szeg\"o} kernel 
\begin{equation*}
K(x^\prime_1, \dots, x^\prime_q; x^{\prime\prime}_1, \dots, x^{\prime\prime}_q) 
= \frac{1}{\prod_{t=1}^d (1 - x^\prime_{1,t}  \cdots x^\prime_{q,t}
\overline{x^{\prime\prime}_{1,t}} 
\cdots \overline{x^{\prime\prime}_{q,t}})}.
\end{equation*}
This kernel generates a reproducing kernel Banach space of 
multi-variable analytic functions \cite{Paul,You01}
\begin{equation*}
B_{\rho, 0}^r(\XC) = \bigg\{ f \in \C^\XC \,\bigg\vert\, 
(\exists\, (c_\nu)_{\nu \in {\N}^d} \in \ell^r(\N^d;\C))(\forall\, x \in \XC) 
\Big(f(x) = \sum_{\nu \in \N^d} 
 c_\nu  x^\nu \Big) \bigg\}
\end{equation*}
with norm
$\norm{f}_{B_{\rho,b}^r(\XC)} =  \norm{c}_r$,
where ${(c_\nu)}_{\nu \in \N^d} \in 
\ell^r(\N^d;\C)$ is such that $f=\sum_{\nu\in \N^d}  c_\nu  x^{\nu}$ (pointwise).
This space reduces to the Hardy space when $r=2$.
Moreover, \eqref{eq:20160217e} yields the following homogenous polynomial form
\begin{equation*}
u \in \C^n \mapsto \bigg\lVert \sum_{i=1}^n u_i \Phi(x_i) \bigg\rVert^{r^*}_{r^*} =
\sum_{\substack{\alpha \in \N^n, \beta \in \N^n\\ \abs{\alpha}=q, \abs{\beta}=q}} 
\binom{q}{\alpha}\binom{q}{\beta}
\frac{u^\alpha \overline{u}^\beta}
{\prod_{t=1}^d (1 - (\overline{x}_{\cdot, t})^\alpha (x_{\cdot, t})^\beta)}.
\end{equation*}
Finally, in view of \eqref{eq:20160217f}, we have the following tensor-kernel representation
\begin{equation*}
\pair{w}{\Phi(x)}_{r,r^*}  = \frac{1}{n^{m-1}}
\sum_{\substack{i_1,\dots, i_{q} = 1\\ \ii_1,\dots, \ii_{q-1} = 1}}^n  
\frac{u_{i_1} \cdots u_{i_{q}}
\overline{u_{\ii_1}} \cdots \overline{u_{\ii_{q-1}}}}
{\prod_{t=1}^d (1 -  x_{\ii_1,t} \cdots x_{\ii_{q-1},t} x_t, 
\overline{x}_{i_1,t} \cdots \overline{x}_{i_q,t})}.
\end{equation*}
\item Set $(\gamma_k)_{k \in \N}\equiv {((k+1)/\pi)}_{k \in \N}$ in \eqref{eq:case2C}.
Then $R_\gamma = 1$ and
$\psi(\zeta) = 1/(\pi(1 - \zeta)^2)$. We then obtain the following Taylor type tensor kernel
\begin{equation*}
K(x^\prime_1, \dots, x^\prime_q; x^{\prime\prime}_1, \dots, x^{\prime\prime}_q) 
= \frac{1}{ \pi^d \prod_{t=1}^d {(1 - x^\prime_{1,t}  \cdots x^\prime_{q,t}
\overline{x^{\prime\prime}_{1,t}} 
\cdots \overline{x^{\prime\prime}_{q,t}})}^2}.
\end{equation*}
This kernel gives rise to a
reproducing kernel Banach space of analytic functions
which reduces to the Bergman space when $m=2$. 
Proceeding as in the previous point, the expression of the corresponding
 homogeneous polynomial form and the
representation formula can be obtained.
\item Let
$(\gamma_k)_{k \in \N} = {\big(1/k!\big)}_{k \in \N}$ in
\eqref{eq:case2C}.
Then $R_\gamma=+\infty$ and $\psi(\zeta) = e^\zeta$. Hence, by \eqref{eq:20160217a},
\begin{equation*}
K(x^\prime_1,\dots, x^\prime_q;
x^{\prime\prime}_1,\dots, x^{\prime\prime}_q) 
= \prod_{t=1}^d e^{x^\prime_{1,t}\cdots x^\prime_{q,t}
\overline{x^{\prime\prime}_{1,t}}\cdots \overline{x^{\prime\prime}_{q,t}}},
\end{equation*}
which is the tensor-\emph{exponential kernel} and
the  form \eqref{eq:20160217e} becomes
\begin{equation*}
u \in \C^n \mapsto \bigg\lVert \sum_{i=1}^n u_i \Phi(x_i) \bigg\rVert^{r^*}_{r^*} =
\sum_{\substack{\alpha \in \N^n, \beta \in \N^n\\ \abs{\alpha}=q, \abs{\beta}=q}} 
\binom{q}{\alpha}\binom{q}{\beta}e^{\sum_{j =1}^d (\overline{x}_{\cdot, j})^\alpha 
(x_{\cdot, j})^\beta} 
u^\alpha \overline{u}^\beta.
\end{equation*}
The corresponding tensor
representation is
\begin{equation*}
\pair{w}{\Phi(x)}_{r,r^*} = \frac{1}{n^{m-1}}
\sum_{\substack{i_1,\dots, i_{q} = 1\\ \ii_1,\dots, \ii_{q-1} = 1}}^n  
\prod_{t=1}^d e^{x_{i_1,t} \cdots x_{i_{q-1},t} x_t, 
\overline{x}_{\ii_1,t} \cdots \overline{x}_{\ii_q,t}}.
\end{equation*}
\item Let $\alpha>0$, set 
\begin{equation*}
(\forall\, k \in \N) \qquad \gamma_k = \binom{-\alpha}{k}(-1)^k = 
\prod_{i=1}^k\frac{\alpha + i - 1}{i}>0,
\end{equation*}
and define $(\rho_\nu)_{\nu \in \N^d}$ according to \eqref{eq:case1C}.
Then $R_\gamma = 1$ and $\psi(z) = (1 - \zeta)^{-\alpha}$ and formula
\eqref{eq:20160301j} yields the
following tensorial version of the \emph{binomial kernel} \cite{SteChi2008}
\begin{equation*}
K(x^\prime_1,\dots, x^\prime_q;
x^{\prime\prime}_1,\dots, x^{\prime\prime}_q) 
=  \frac{1}{
 \Big(1 - \sum_{t=1}^d x^\prime_{1,t}  \cdots x^\prime_{q,t}
\overline{x^{\prime\prime}_{1,t} }
\cdots \overline{x^{\prime\prime}_{q,t}} \Big)^{\alpha}}.
\end{equation*}
\item Let $s \in \N$, set
\begin{equation*}
(\forall\, k \in \N)\qquad \gamma_k =
\begin{cases}
\dbinom{s}{k} &\text{if}\ k \leq s\\[2ex]
0&\text{if}\ k>s,
\end{cases}
\end{equation*}
and define $(\rho_\nu)_{\nu \in \N^d}$ according to \eqref{eq:case1C}.
Then $R_\gamma = +\infty$
and $\psi (\zeta) = (1 + \zeta)^s$. This way, by \eqref{eq:20160301j}, we have
\begin{equation*}
K(x^\prime_1,\dots, x^\prime_q;
x^{\prime\prime}_1,\dots, x^{\prime\prime}_q) 
=  
 \Big(1 + \sum_{t=1}^d x^\prime_{1,t}  \cdots x^\prime_{q,t}
\overline{x^{\prime\prime}_{1,t} }
\cdots \overline{x^{\prime\prime}_{q,t}} \Big)^{s},
\end{equation*}
which is the \emph{polynomial tensor-kernel} of order $s$.
By \eqref{eq:case1C} we have that $\rho_\nu > 0$ if $\abs{\nu}\leq s$
and $\rho_\nu = 0$ if $\abs{\nu}>s$. Therefore, recalling \eqref{eq:20160303b}, we have that 
\begin{equation*}
B_{\rho, 0}^r(\XC) = \bigg\{ f \in \C^\XC \,\bigg\vert\, 
(\exists\, (c_\nu)_{\nu \in {\N}^d} \in \ell^r(\N^d;\C))(\forall\, x \in \XC) 
\Big(f(x) = \sum_{\nu \in \N^d} 
 c_\nu  \phi_\nu(x) \Big) \bigg\},
\end{equation*}
is the space
of polynomials in $d$ variables with coefficients in $\C$ of degree up to $s$.
\end{enumerate}
\end{example}

\section{Conclusion}
\label{sec:con}

In this work 
we first provided a complete duality theory for 
support vector regression in Banach function spaces with general regularizers.
Then, we specialized the analysis to reproducing kernel Banach spaces
that admit a representation in terms of a (countable) dictionary of functions
with $\ell^{r}$-summable coefficients and regularization terms of type 
$\varphi(\norm{\cdot}_r)$, being $r=m/(m-1)$ and $m$ an even integer.
In this context we showed that the problem of support vector regression can be  
explicitly solved through the introduction of a new type of kernel of tensorial type 
(with degree $m$)
which completely encodes the finite dimensional dual problem as well as the representation
of the corresponding infinite dimensional primal solution (the regression function).
This can provide a new and effective computational framework for solving 
support vector
regression in Banach space setting. We finally study
a whole class of reproducing kernel Banach spaces of analytic functions to which
the theory applies and show significant examples which can become useful in applications.

\hspace{2ex}

{\bf Acknowledgments.} {\small The research leading to these results has received funding from the European Research Council (FP7/2007--2013) / ERC AdG A-DATADRIVE-B (290923) under the European Union's Seventh Framework Programme. This paper reflects only the authors' views, the Union is not liable for any use that may be made of the contained information; Research Council KUL: GOA/10/09 MaNet, CoE PFV/10/002 (OPTEC), BIL12/11T; PhD/Postdoc grants; Flemish Government: FWO: PhD/Postdoc grants, projects: G.0377.12 (Structured systems), G.088114N (Tensor based data similarity); IWT: PhD/Postdoc grants, projects: SBO POM (100031); iMinds Medical Information Technologies SBO 2014; Belgian Federal Science Policy Office: IUAP P7/19 (DYSCO, Dynamical systems, control and optimization, 2012--2017).}

\appendix
\section{Proofs of section~\ref{sec:GenSVR}}

\begin{proof}[proof of Theorem~\ref{thm:pridu}]
The Banach spaces $\displaystyle L^p(P)$ and $L^{p^*}(P)$ are put in duality by means
of the pairing
\begin{equation}
\pair{\cdot}{\cdot}_{p,p^*}\colon L^p(P)\times L^{p^*}(P) \to \R
\colon (e,u) \mapsto \int_{\XC\times\YC} e(x,y) u(x,y) \ud P(x,y).
\end{equation}
In virtue of {\bf A3}, the following linear operator
\begin{equation}
A\colon \FF\times \R \to L^p(P)\ \text{s.t.}\ 
(\forall\, (w,b) \in \FF\times\R)\ 
A(w,b)\colon \XC\times\YC \to \R\colon (x,y) \mapsto \pair{w}{\Phi(x)} + b
\end{equation}
is well-defined and the function
\begin{equation*}
\mathrm{pr}_2\colon \XC\times\YC\to \R \colon (x,y) \mapsto y,
\end{equation*}
is in $\displaystyle L^{p}(P)$. 
Then problem \eqref{eq:mainprob1} can be written in the following
constrained form
\begin{equation}
\label{eq:mainprob2}
\left[
\begin{aligned}
&\min_{(w,b,e) \in \FF\times\R \times L^p(P) }
\gamma \int_{\XC\times\YC} \loss(e(x,y))\ud P(x,y) + G(w),\\[1ex]
&\text{subject to}\ 
\mathrm{pr}_2 - A (w,b) = e
\end{aligned}
\right.
\end{equation}
--- where, in the constraint, the equality is meant to be in $L^p(\PP)$ ---
and hence $(\mathcal{P})$ follows.
Now, define the following integral functional
\begin{equation*}
R_{\PP} \colon L^p(P) \to \R \colon e \mapsto \int_{\XC \times\YC} \loss(e(x,y))\ud P(x,y),
\end{equation*}
the linear operator
\begin{equation*}
B\colon \FF\times \R \times L^p(\PP) \to L^p(\PP)\colon (w,b,e)\mapsto A (w,b) + e,
\end{equation*}
and the functional
\begin{equation*}
F\colon \FF\times\R \times L^p(\PP) \to \RX\colon (w,b,e)\mapsto \gamma R_\PP(e) + G(w).
\end{equation*}
We note that the functional $R_\PP$ is well defined, convex, and continuous. This follows from
from the convexity and continuity of $\loss$ and from 
the fact that, because of {\bf A1},
for every $(x,y) \in \XC\times\YC$,
 $\loss(e(x,y)) \leq a  + b \abs{e(x,y)}^p$.
Then, problem \eqref{eq:mainprob2} can be equivalently written as
\begin{equation}
\label{eq:mainprob4}
\min_{(w,b,e) \in \FF \times \R\times  L^p(\PP)} F(w,b,e) + \iota_{\{-\mathrm{pr}_2 \}} (-B(w,b,e)).
%\quad\text{with}\  f(w,b,e) = \gamma R(e) + G(w).
\end{equation}
This form of problem \eqref{eq:mainprob1} is amenable by the 
Fenchel-Rockafellar duality theory.
In view of Fact~\ref{fact:FRduality} we need only to check that
$0 \in \mathrm{int}\big(-B(\dom F) + \dom \iota_{\{-\mathrm{pr}_2\}} \big)$. This is 
almost immediate. 
Indeed, since 
$\dom F = \dom G \times\R\times L^p(\PP)$, we have
\begin{equation*}
B(\dom F) = \big\{ A (w,b)+ e \,\big\vert\, (w,b) 
\in \dom G\times\R\ \text{and}\ e \in L^p(\PP) \big\} = L^p(\PP).
\end{equation*}
Now we compute the dual of \eqref{eq:mainprob4}. We have
\begin{equation}
\label{eq:g*}
(\forall\, u \in L^{p^*}(P))\qquad (\iota_{\{-\mathrm{pr}_2\}})^*(u) 
= \pair{-\mathrm{pr}_2}{u}_{p,p^*}
\end{equation}
and, for every $(w^*,b^*, u) \in \FF^*\!\times \R\times L^{p^*}(\PP)$,
\begin{equation}
\label{eq:f*}
\begin{aligned}
F^*(w^*,b^*,u) &= \sup_{(w,b,e) \in \FF\times\R\times L^p(P)} 
\pair{(w,b,e)}{(w^*,b^*,u)} - F(w,b,e)\\[1ex]
&= \sup_{w \in \FF} \sup_{b \in \R} \sup_{e \in L^p(\PP)} \pair{w}{w^*} - G(w)
+ \pair{u}{e}_{p,p^*} - \gamma R_\PP(e) + b b^*\\[1ex]
&= 
\begin{cases}
G^*(w^*) + \gamma R_{\PP}^*(u/\gamma) &\text{if } b^* =0\\
+\infty &\text{if } b^* \neq 0.
\end{cases}
\end{aligned}
\end{equation}
Moreover, we need also to compute 
$\displaystyle A^*\colon L^{p^*}(\PP) \to \FF^*\!\times \R$ and
$\displaystyle B^*\colon L^{p^*}(\PP) \to \FF^*\!\times \R \times L^{p^*}(\PP)$.
To that purpose, we note that
for every $\displaystyle (w,b,e) \in \FF\times\R \times L^p(\PP)$ and every 
$\displaystyle u \in L^{p^*}(\PP)$,
\begin{equation*}
\begin{aligned}
\pair{B(w,b,e)}{u}_{p,p^*} &= \pair{A (w,b) + e}{u}_{p,p^*} 
= \pair{(w,b)}{A^* u} + \pair{e}{u}_{p,p^*}\\
&= \pair{(w,b,e)}{(A^* u, u)}
\end{aligned}
\end{equation*}
and
\begin{equation*}
\begin{aligned}
\pair{(w,b)}{A^* u} &= \pair{A (w,b)}{u}_{p,p^*} \\[1ex]
&= \int_{\XC\times\YC} (\pair{w}{\Phi(x)} + b) u(x,y) \ud P(x,y)\\
&= \Big\langle w, \int_{\XC\times\YC} u(x,y) \Phi(x)  \ud P(x,y) \Big\rangle
+ b \int_{\XC\times\YC} u\ud P\\
&=\Big\langle (w,b), \bigg( \int_{\XC\times\YC} u(x,y) \Phi(x)  \ud P(x,y), 
\int_{\XC\times\YC} u \ud P \bigg)\Big\rangle,
\end{aligned}
\end{equation*}
which yields
\begin{equation}
\label{eq:A*}
A^* u = \bigg( \int_{\XC\times\YC} u \Phi  \ud P, 
\int_{\XC\times\YC} u \ud P \bigg)
\end{equation}
and
\begin{equation}
\label{eq:B*}
B^* u = (A^* u, u) = 
\bigg( \int_{\XC\times\YC} u \Phi  \ud P, 
\int_{\XC\times\YC} u \ud P,\, u \bigg),
\end{equation}
where, for brevity, we put $\int_{\XC\times\YC} u \Phi \ud \PP
= \int_{\XC\times\YC} u(x,y) \Phi(x) \ud \PP(x,y)$.
Thus, taking into account \eqref{eq:f*},\eqref{eq:A*}, and \eqref{eq:B*}, we have that,
for every $u \in L^{p^*}(\PP)$,
\begin{equation*}
F^*(B^* u) = F^* (A^* u, u) 
=
\begin{cases} 
\displaystyle
G^*\bigg( \int_{\XC\times\YC} u\Phi \ud P \bigg) 
+ \gamma R^*_\PP(u/\gamma) &\text{if }
\displaystyle \int_{\XC\times\YC} u \ud P = 0\\[2ex]
+\infty &\text{otherwise.}
\end{cases}
\end{equation*}
Moreover, it follows from \cite[Theorem~21(a)]{Rock74} that
the Fenchel conjugate of $R_\PP$ is still an integral operator, more precisely
\begin{equation*}
(\forall\, u \in L^p(\PP))\qquad R^*_\PP(u/\gamma) = \int_{\XC \times \YC} L^*(u(x,y)/\gamma) \ud \PP(x,y).
\end{equation*}
Therefore,  recalling \eqref{eq:g*}, the final form $(\mathcal{D})$ is obtained.
The corresponding optimality conditions for problem \eqref{eq:mainprob4}
and its dual $(\mathcal{D})$ are (see Fact~\ref{fact1})
\begin{equation}
\label{eq:optSVM0}
B^* u \in \partial F(w,b, e) = \partial G(w) \times\{0\} \times \gamma \partial R(e)
\quad\text{and}\quad B(w,b, e) = \mathrm{pr}_2.
\end{equation}
Now, recalling \eqref{eq:B*},
conditions \eqref{eq:optSVM0}
can be gathered together as follows
\begin{equation}
\label{eq:optSVM}
\begin{cases}
\displaystyle \int_{\XC\times\YC} u\Phi \ud \PP \in \partial G(w)\\[2ex]
\displaystyle \int_{\XC\times\YC} u \ud \PP = 0\\[2.5ex]
\dfrac{u}{\gamma} \in   \partial R(e)\\[2ex]
y - \pair{w}{\Phi(x)} - b  = e(x,y)
\quad \text{for $\PP$-a.a.}\ (x,y) \in \XC\times\YC.
\end{cases}
\end{equation}
Thus, subdifferentiating under the integral sign \cite[Theorem~21(c)]{Rock74}
and recalling \eqref{eq:20160302a}, \eqref{eq:optPD} follows.
\end{proof}

\begin{proof}[proof of Corollary~\ref{cor:20160223m}]
Let $t>0$. We first note that, since $0$ is the unique minimizer of $\varphi$
and $t>0$, then $0 \notin \partial\varphi(t)$; moreover,
for every $\xi \in \partial\varphi(t)$, we have $\xi t \geq \varphi(t) - \varphi(0)>0$,
hence,  $\xi>0$. Now, if $w=0$, then \eqref{eq:20160223l} holds
trivially. Suppose that $w\neq 0$. Then, it follows from Fact~\ref{fact1} that
%when $w\neq 0$,
\begin{equation*}
\partial G(w) =
\frac{\partial \varphi(\norm{w})}{\norm{w}^{r-1}} J_r(w).
\end{equation*}
Therefore,
 it follows from the first of \eqref{eq:optPD} and 
 % \eqref{eq:optSVM} 
\eqref{eq:20160302a} that
\begin{equation*}
\int_{\XC\times\YC} u\Phi \ud \PP = \frac{\xi}{\norm{w}^{r-1}} J_r(w),
\qquad \xi \in \partial \varphi(\norm{w}).
\end{equation*}
Hence, since $\xi > 0$,
\begin{equation*}
J_r(w) = \frac{\norm{w}^{r-1}}{\xi} \int_{\XC\times\YC} u\Phi \ud \PP
\end{equation*}
and the statement follows.
\end{proof}


\begin{thebibliography}{15}

\bibitem{Livre1} 
H.~H.~Bauschke and P.~L.~Combettes,
{\em Convex Analysis and Monotone Operator Theory in Hilbert 
Spaces.}
Springer, New York 2011.

\bibitem{Bog07} 
V. I. Bogachev,
{\em Measure Theory.}
Springer, Berlin 2007.

\bibitem{Bona11} 
M.~B\'ona,
{\em A Walk Through Combinatorics. 3rd Ed.}
World Scientific, Singapore 2011.

\bibitem{Ciora90}
I.~Cioranescu,
{\em Geometry of Banach Spaces, Duality Mappings and Nonlinear
Problems.} 
Kluwer, Dordrecht 1990.

\bibitem{Com15} 
P.~L.~Combettes, S.~Salzo, and S.~Villa,
Consistency of Regularized Learning Schemes in Banach Spaces.
{\em Anal. App.,} published online, 2016-12-07.
%arXiv:1410.6847v3, 2015. 

\bibitem{Com15b} 
P.~L.~Combettes, S.~Salzo, and S.~Villa,
Consistent Learning by Composite Proximal Thresholding.
{\em Math. Program., Series B} to appear, 2017.

\bibitem{Cri00} 
N.~Cristianini and J.~Shawe-Taylor,
{\em An Introduction to Support Vector Machines.}
Cambridge University Press, Cambridge 2000.

\bibitem{Demo09}
C. De Mol, E. De Vito, and L. Rosasco, 
Elastic-net regularization in learning theory, 
{\em J. Complexity,}
vol. 25, pp. 201--230, 2009. 

\bibitem{DeVito04}
E. De Vito, L. Rosasco, A. Caponnetto, M. Piana, and A. Verri, 
Some properties of regularized kernel methods,
{\em J. Mach. Learn. Res.,} 
vol. 5, pp. 1363--1390, 2004.

\bibitem{Fass15}
G. E.~Fasshauer, F. J.~Hickernell, and Q.~Ye, 
Solving support vector machines in reproducing kernel Banach spaces 
with positive definite functions, 
{\em Appl. Comput. Harmon. Anal.,}
vol. 38, pp. 115--139, 2015. 

\bibitem{Fu1998}
W.~Fu,
Penalized regressions: the bridge versus the lasso,
{\em J. Comput. Graph. Stat.,} 
vol. 7, pp. 397--416, 1998.

\bibitem{Gir98}
F.~Girosi, 
An Equivalence Between Sparse Approximation and Support Vector Machines, 
{\em Neural Comput.,}
vol. 10(6), pp. 1455--1480, 1998. 

\bibitem{Livre2} 
J.-B. Hiriart-Urruty and C.~Lemar\'echal,
{\em Convex Analysis and Minimization Algorithms II.}
Springer, Berlin 1996.

\bibitem{Hofm08}
T.~Hofmann, B.~Sch\"olkopf, and A.~J.~Smola, 
Kernel methods in machine learning,
{\em Ann. Statist.,}
vol. 36, pp. 1171--1220, 2008.

\bibitem{Kol2009}
V. Koltchinskii,
Sparsity in penalized empirical risk minimization,
{\em Ann. Inst. Henri Poincar\'e Probab. Stat.,} 
vol. 45, pp. 7--57, 2009.

\bibitem{Men10}
B.~S.~Mendelson and J.~Neeman,
Regularization in Kernel Learning,
\emph{Ann. Statist.}, 
38(1), pp. 526--565, 2010.

\bibitem{Paul}
V.~I.~Paulsen,
\emph{An Introduction to the Theory of Reproducing Kernel Hilbert Spaces.}
[On line]. Available: \url{http://www.math.uh.edu/~vern/rkhs.pdf}


\bibitem{Rock74} 
R. T. Rockafellar,
{\em Conjugate Duality and Optimization.}
SIAM, Philadelphia, PA 1974.

\bibitem{Sch01}
B.~Sch\"olkopf, R.~Herbrich, and A.~J.~Smola, 
A Generalized Representer Theorem. In 
{\em Computational Learning Theory: 14th Annual Conference on Computational Learning Theory, COLT 2001}.
Springer Berlin Heidelberg, 2001.

\bibitem{SteChi2009} 
I. Steinwart and A. Christmann,
Sparsity of SVMs that use the $\varepsilon$-insensitive loss. In 
{\em Advances in Neural Information Processing Systems 21}.
Curran Associates, Inc., 2009.

\bibitem{SteChi2008} 
I. Steinwart and A. Christmann,
{\em Support Vector Machines.}
Springer, New York 2008.

\bibitem{Sri11} 
B.~K.~Sriperumbudur, K.~Fukumizu, and G.~R.~G.~Lanckriet,
Learning in Hilbert vs.~Banach spaces: a measure embedding viewpoint, in:
\emph{Advances in Neural Information Processing Systems 24.}
Curran Associates, Inc., 2011.


\bibitem{Suy2002} 
J.~A.~K.~Suykens, T.~Van~Gestel, J.~De~Brabanter, B.~De~Moor, J.~Vandewalle,
{\em Least Squares Support Vector Machines.}
World Scientific, Singapore 2002.

\bibitem{Vap98}
V. N. Vapnik, 
{\em Statistical Learning Theory.}
Wiley, New York 1998.

\bibitem{You01}
R.~M.~Young, 
{\em An Introduction to Nonharmonic Fourier Series.}
Academic Press, San Diego 2001.


\bibitem{Zali02}
C. Z\u{a}linescu,
{\em Convex Analysis in General Vector Spaces.}
World Scientific, River Edge, NJ 2002.

\bibitem{Zhang2009}
H. Zhang, Y. Xu, and J. Zhang,
Reproducing kernel Banach spaces for machine learning,
{\em J. Mach. Learn. Res.,} 
vol. 10, pp. 2741--2775, 2009.

\bibitem{Zhang2012}
H. Zhang and J. Zhang, 
Regularized learning in Banach spaces as an optimization
problem: representer theorems,
{\em J. Global Optim.,} 
vol. 54, pp. 235--250, 2012.

\bibitem{Zwi2009}
B.~Zwicknagl, 
Power series kernels,
{\em Constr. Approx.,} 
vol. 29, pp. 61--84, 2009.

\end{thebibliography}
\end{document}